\documentclass{article}

\usepackage[english]{babel}
\usepackage{amsthm}
\usepackage{amsmath}
\usepackage{amssymb}
\usepackage{graphicx}
\usepackage{enumitem}
\usepackage[colorlinks]{hyperref}
\usepackage{caption}
\usepackage{subcaption}
\usepackage{pict2e,color,colordvi,amscd}
\usepackage{indentfirst}
\usepackage{tikz}
\usepackage{float}

\newtheorem{definition}{Definition}[section]
\newtheorem{theorem}[definition]{Theorem}

\newtheorem{lemma}[definition]{Lemma}
\newtheorem{claim}[definition]{Claim}

\newtheorem{fact}{Fact}

\newtheorem{conjecture}[definition]{Conjecture}

\title{Bipartite graphs with the double Hall property}
\date{\relax}
\author{
    Guantao Chen\thanks{Department of Mathematics and Statistics, Georgia State University, Atlanta, GA 30303, gchen@gsu.edu. Research of this author was partially supported by NSF grant DMS-2154331.} \and
    Mikhail Lavrov\thanks{misha.p.l@gmail.com.} \and
    Yuying Ma\thanks{Department of Mathematics and Statistics, Georgia State University, Atlanta, GA 30303, yma29@gsu.edu.} \and
    Yimo Su\thanks{Department of Mathematics and Statistics, Georgia State University, Atlanta, GA 30303, ysu12@gsu.edu.} \and
    Jennifer Vandenbussche\thanks{Department of Mathematics, Kennesaw State University, Marietta, GA 30060, jvandenb@kennesaw.edu.}
}

\begin{document}
\maketitle

\begin{abstract}
The \textit{super-neighborhood} of a vertex set $A$ in a graph $G$, denoted by $\mathsf{\Lambda}^2(A)$,  is the set of vertices adjacent to  at least two vertices in $A$.  We say that a bipartite graph $G=(X, Y)$ with $|X| \geq 2$ satisfies the {\em double Hall property} (with respect to $X$) if $|\mathsf{\Lambda}^2(A)| \ge |A|$ for any subset $A \subseteq X$ with $|A| \geq 2$.
Kostochka et al. first conjectured that if a bipartite graph $G=(X, Y)$ satisfies a slightly weaker version of the double Hall property, then $G$ contains a cycle that covers all vertices of $X$. They verified their conjecture for $|X|\le 6$. In this paper, we extend their result to $|X| =7$. Later, Salia conjectured that every bipartite graph satisfying the double Hall property has a cycle covering all vertices of $X$. We show that Salia's conjecture is almost equivalent to a much weaker conjecture requiring vertices in $Y$ to have high degrees. By extending a result of Bar\'at et al., we also show that Salia’s conjecture holds for some graphs where the vertices of $Y$ have degree either $2$ or very high.
Finally, we establish a lower bound for the maximum degree of graphs satisfying the double Hall property and present deterministic and probabilistic constructions of such graphs that approach this bound.

\vspace{0.3cm}
\end{abstract}

\section{Introduction} 
In this paper, all graphs considered are {\em simple graphs}  -- graphs with a finite number of vertices and with no loops or parallel edges. Let $G$ be a graph with vertex set $V(G)$ and edge set $E(G)$. For a vertex $v \in V(G)$, we write  $\mathsf{\Lambda}_G(v)$ for the set of vertices in $G$ adjacent to $v$, and set $\deg_G(v) := |\mathsf{\Lambda}_G(v)|$.
The graph $G$ is {\em $d$-regular} if $\deg_G(v) = d$ for all $v \in V(G)$.
For any vertex set $S\subseteq V(G)$, the neighborhood $\mathsf{\Lambda}_G(S)$ of 
$S$ is the set of vertices in $G$ adjacent to an element of $S$. More generally, we define $\mathsf{\Lambda}^i_G(S)$ to be the set of all vertices in $G$ that are adjacent to at least $i$ elements of $S$. We omit the subscript when the graph $G$ is clear from context. In addition, for any set $S\subseteq V(G)$, we denote $G[S]$ as the subgraph of $G$ induced by $S$.

A bipartite graph with partitions $X$ and $Y$ is called a {\em bigraph} if the ordered pair $(X,Y)$ of its vertex partition is fixed. This paper discusses properties of bigraphs satisfying two related neighborhood conditions.
A bigraph $G=(X, Y)$ with $|X| \ge 3$ is said to satisfy the {\em super-neighborhood property} (or ``$G$ is snp'' for short) if, for each $S \subseteq X$ with $|S| \geq 3$, $|\mathsf{\Lambda}^2(S)| \geq |S|$ and $G[S \cup \mathsf{\Lambda}^2(S)]$ is $2$-connected.

The super-neighborhood property was first proposed by Kostochka et al.~\cite{superneighbor}, and is closely tied to the concept of super-pancyclicity in hypergraph theory. A {\em hypergraph} $H = (V(H), E(H))$ consists of a vertex set $V(H)$ and an edge set $E(H)$, where each edge is a subset of $V(H)$. A hypergraph $H$ is said to be {\em super-pancyclic} if, for every subset $A \subseteq V(H)$ with $|A| \ge 3$, $H$ contains a Berge cycle whose base vertex set is $A$. The {\em incidence graph} of a hypergraph $H$ is a bigraph $G$ with partitions $X = V(H)$ and $Y = E(H)$, and an edge in $G$ between $x \in X$ and $y \in Y$ exists if and only if $y$ is an edge in $H$ containing the vertex $x$. This bipartite representation allows us to study hypergraph properties through their incidence graphs. We call $G$ \emph{($X$-)supercyclic} if, for every subset $X' \subseteq X$ with $|X'| \ge 3$, $G$ contains a cycle $C$ such that $V(C) \cap X = X'$. Clearly, a hypergraph $H$ is super-pancyclic if and only if its incidence graph $G$ is supercyclic. This equivalence underscores the importance of incidence bigraphs in studying super-pancyclic hypergraphs. In this context, it is easy to see that every supercyclic bigraph must satisfy the super-neighborhood property. In~\cite{superneighbor}, Kostochka et al. conjecture that the converse also holds:

\begin{conjecture}[Kostochka et al.~\cite{superneighbor}]\label{con-KLLZ21} If a bigraph $G=(X,Y)$ is snp, then there is a cycle containing all vertices of $X$.    
\end{conjecture}

In the same paper~\cite{superneighbor}, they prove the following result.
\begin{theorem}[Kostochka et al.~\cite{superneighbor}]\label{thm:KLLZ21-a}
If a bigraph $G=(X,Y)$ is snp and $|X| \in \{3, 4, 5, 6\}$, then $G$ has a cycle containing all vertices of $X$.  
\end{theorem}

It is also natural to include sets $S$ of size $2$. Accordingly, we say that a bigraph $G=(X, Y)$ with $|X| \ge 2$ satisfies the {\em double Hall property} (or ``$G$ is dHp'' for short) if $|\mathsf{\Lambda}^2(S)| \geq |S|$ for every subset $S \subseteq X$ with $|S| \geq 2$.

It is not difficult to see that every $dHp$ bigraph with $|X| \ge 3$ is also $snp$. The following conjecture made later by Salia is a weakening of Conjecture \ref{con-KLLZ21}.   
\begin{conjecture}[Salia~\cite{salia2021extremal}] \label{con-Salia}
Every dHp bigraph $G = (X,Y)$ with $|X| \ge 2$ has a cycle containing all vertices of $X$. 
\end{conjecture}

Recently, Bar\'at et al.~\cite{Hungarian} proved the following.

\begin{theorem}[Theorem~3.2 in~\cite{Hungarian}]\label{thm:hungarian-cycle-cover}
Any dHp bigraph $G = (X, Y)$ has a collection of disjoint cycles whose union covers all vertices of $X$. 
\end{theorem}

In this paper, we present progress on these conjectures.  We show that Conjecture~\ref{con-KLLZ21} holds for $|X|=7$.  
 
\begin{theorem}\label{thm:cycle-for-7}
Let $G=(X,Y)$ be a bigraph with $3 \le |X| \le 7$. If $G$ is snp, then $G$ has a cycle covering all vertices of $X$.
\end{theorem}

Also, we introduce two conjectures that are slightly weaker than Conjecture~\ref{con-Salia}, and we show that these two conjectures are equivalent to each other. 

\begin{conjecture}\label{conj:Big-degree-Y}
Let $k$ be a nonnegative integer, and consider a dHp bigraph $G=(X, Y)$. If $|X| = n > \max \{2k+1, k(k+1)\}$ and $\deg(y) \ge n - k$ for all $y \in Y$, then there is a cycle in $G$ covering all vertices of $X$.
\end{conjecture}

\begin{conjecture}\label{subconj:Big-degree-Y}
If a dHp bigraph $G = (X, Y)$ satisfies $|\mathsf{\Lambda}_G(X)| \geq |X| + 1$, then there exists a collection of disjoint $Y-Y$ paths whose union covers all of $X$.
\end{conjecture}

Leveraging Theorem~\ref{thm:cycle-for-7}, we can show that Conjecture~\ref{conj:Big-degree-Y} holds for $k \le 7$. This result is formally presented below.
\begin{theorem}\label{thm:Big-degree-Y-for-7}
For any integer $k$ with $0 \le k \le 7$ and any dHp bigraph $G$, if $|X| = n > \max\{2k+1, k(k+1)\}$ and $\deg(y) \ge n - k$ for all $y \in Y$, then  there is a cycle in $G$ covering all vertices of $X$.
\end{theorem}

In addition to Theorem~\ref{thm:hungarian-cycle-cover}, in~\cite{Hungarian}, Bar\'at et al.\ also prove that any dHp bigraph $G=(X,Y)$ with $|X| \geq 2$ guarantees a cycle covering $X$ if, for any $y \in Y$, $\deg(y) \in \{2, |X|\}$. We give the following extended theorem.

\begin{theorem}\label{thm:set-of-degree-Y}
Let $G = (X, Y)$ be a dHp bigraph with $|X| \ge 2$. Then $G$ has a cycle covering all vertices in $X$ if, for any $y \in Y$, $\deg(y) \in \{2, |X|-2, |X|-1, |X|\}$.
\end{theorem}

When considering partial progress towards these open conjectures by adding degree conditions, it is natural to ask: what can we say about the possible vertex degrees in a dHp bigraph? Of course, the double Hall property only becomes easier to satisfy with the addition of edges, so we are most interested in finding dHp bigraphs where degrees are low.

To maintain low degree in $Y$ alone, there is a natural construction of a dHp bigraph with $|X|=n$ and $|Y|=n(n-1)$ where the maximum degree in $Y$ is $2$. To construct this graph, for every pair $x_i, x_j \in X$, add two vertices $y_{ij}^1, y_{ij}^2$ to $Y$ with edges only to $x_i$ and $x_j$. In this construction, although vertices in $Y$ have very low degree, all vertices in $X$ have degree $2(n-1)$.

It is also possible to construct a dHp bigraph where the maximum degree in $X$ is low. In~\cite{Hungarian}, Bar\'at et al.\ construct a dHp bigraph with $|X|=n=2^k$ and a maximum degree in $X$ of $k+1$. However, in this construction, two vertices in $Y$ have degree $n$, and other vertices in $Y$ unavoidably have degree linear in $n$.

This leaves a natural follow-up question: what is the smallest maximum degree we can achieve if we consider vertices in both $X$ and $Y$? We prove a lower bound on the maximum degree.

\begin{theorem}\label{thm:dhp-maximum-degree-bound}
    Let $G = (X, Y)$ be a dHp bigraph with $|X| = n \ge 2$ and maximum degree $\Delta(G)=d$. Then $n \le \binom d2 + 1$; it follows that $d \ge \sqrt{2n}$.
\end{theorem}

It is possible to construct graphs in which the inequality of Theorem~\ref{thm:dhp-maximum-degree-bound} is tight. This is a very strong condition on the structure of the graph: $n = \binom d2 + 1$ exactly when $G$ is the incidence graph of a symmetric block design called a {\em biplane}. However, it is conjectured that only finitely many biplanes exist, which means that perhaps we cannot hope for a general construction of this type; see Section~\ref{sec:biplanes} for more details. 

Instead, we use a natural notion of the product of two bipartite graphs to construct additional examples. We prove that the bipartite product of two dHp bigraphs is always dHp. This allows us to find infinitely many low-degree dHp bigraphs, though the maximum degree no longer quite matches the lower bound of Theorem~\ref{thm:dhp-maximum-degree-bound}. Our constructions are summarized in the following theorem. For clarity, recall that a $d$-regular bigraph is one in which every vertex has degree exactly $d$.

\begin{theorem}\label{thm:dhp-maximum-degree-construction}
For $d \in \{2,3,4,5,6,9,11,13\}$, there exists a $d$-regular bigraph $(X, Y)$ satisfying the double Hall property with $|X| = |Y| = n = \binom d2 + 1$ vertices. Furthermore, for infinitely many values of $n$, there exists a $d$-regular bigraph $(X, Y)$ satisfying the double Hall property with $d = n^{\alpha}$, where $\alpha = \log_{79} 13 \approx 0.587$. 
\end{theorem}

We can approach the lower bound more closely with a probabilistic approach. Let $\mathbb G_{n,n,p}$ be the random bigraph with partitions $X$ and $Y$, where $|X|=|Y|=n$ and each edge between $X$ and $Y$ appears independently with probability $p$. We can write any probability $p = p(n)$ in the form $\sqrt{\frac{2\log n + \log \log n + c_n}{n}}$ simply by defining $c_n = np^2 - 2\log n - \log \log n$. Then the following theorem gives the threshold at which $\mathbb G_{n,n,p}$ gains the double Hall property:


\begin{theorem}\label{thm:gnnp-is-dhp}
If $p:= \sqrt{\frac{2\log n + \log \log n + c_n}{n}}$, then the probability that $\mathbb G_{n,n,p}$ has the double Hall property satisfies
\[
    \lim_{n \to \infty} \Pr[\mathbb G_{n,n,p}\text{ is dHp}] =
    \begin{cases}
        0 & c_n \to -\infty, \\
       e^{-e^{-c}} &  c_n \to c, \\
        1 & c_n \to +\infty.
    \end{cases}
\]
\end{theorem}


At this threshold, the average degree is $np \sim \sqrt{2n \log n}$, and with high probability, the maximum degree is also asymptotic to $\sqrt{2n \log n}$. Therefore the random graph $\mathbb G_{n,n,p}$ at the thresholds gives examples of graphs satisfying the double Hall property that only differ from the lower bound of Theorem~\ref{thm:dhp-maximum-degree-bound} by a factor asymptotic to $\sqrt{\log n}$.

Additionally, we recall that Frieze~\cite{FRIEZE1985327} gave the threshold at which $\mathbb G_{n,n,p}$ becomes Hamiltonian (i.e., contains a cycle covering all vertices):
\begin{theorem}[Frieze~\cite{FRIEZE1985327}]
    If $p := \frac{\log n + \log \log n + c_n}{n}$, then the probability that $\mathbb G_{n, n, p}$ is Hamiltonian satisfies
\[ 
\lim_{n \to \infty} \Pr[\mathbb G_{n,n,p}\text{ is Hamiltonian}] =
    \begin{cases}
        0 & c_n \to -\infty, \\
       e^{-2e^{-c}} &  c_n \to c, \\
        1 & c_n \to +\infty.
    \end{cases}
\]
\end{theorem}
Since the threshold in Theorem~\ref{thm:gnnp-is-dhp} asymptotically exceeds that in the above theorem, this implies that $\mathbb G_{n,n,p}$ is Hamiltonian with high probability at this threshold, giving further evidence for Conjecture~\ref{con-Salia}.

Our paper is organized as follows. In the next section, we prove Theorem~\ref{thm:cycle-for-7}. In Section~\ref{section:big-degree-Y-for-7}, we show the equivalence of Conjectures~\ref{conj:Big-degree-Y} and~\ref{subconj:Big-degree-Y} and prove Theorem~\ref{thm:Big-degree-Y-for-7} by applying Theorem~\ref{thm:cycle-for-7}. In Section~\ref{section:set-of-degree-Y} we prove Theorem~\ref{thm:set-of-degree-Y}. Sections~\ref{section:degree-bound-on-dhp} and~\ref{section:dHp-random-graph} are devoted to the possible maximum degree of dHp graphs: in Section~\ref{section:degree-bound-on-dhp}, we give a lower bound on the maximum degree and provide deterministic constructions with low maximum degree, and in Section~\ref{section:dHp-random-graph}, we consider probabilistic constructions and prove Theorem~\ref{thm:gnnp-is-dhp}.

\section{Proof of Theorem~\ref{thm:cycle-for-7}}\label{section:cycle-for-7}

\subsection{Preliminaries}
Let $G=(X,Y)$  be a bigraph with $|X|\geq 3$.
We say $G$ is {\em critical} if the following conditions hold:
\begin{enumerate}[label={(\arabic*)}]
    \item $G$ satisfies the super-neighborhood property but it is not supercyclic,
    \item $\mathsf{\Lambda}^2(X) = Y$, and
    \item $G[X'\cup Y]$ is supercyclic for any proper subset $X'\subset X$ with $|X'|\geq 3$.
\end{enumerate}
 
Moreover, a critical bigraph $G=(X, Y)$ is said to be \emph{saturated critical} if, in addition to being critical, it satisfies the condition that $G + xy$ is supercyclic for any nonadjacent pair $\{x, y\}$ with $x \in X$ and $y \in Y$.
An snp bigraph $G=(X, Y)$ is said to be \emph{snp-minimal} if $G[X \cup Y\backslash\{y\}]$ is not snp for any $y \in Y$.

Conjecture~\ref{con-KLLZ21} states that there is no critical bigraph. However, if there is a critical bigraph, then there is a saturated critical bigraph that is also snp-minimal: Let $G=(X,Y)$ be an snp bigraph that is not supercyclic and minimizes $|X\cup Y|$. Since the complete bipartite graph with partitions $X$ and $Y$ is supercyclic, there is a bigraph $G'=(X, Y)$ with $G'\supseteq G$ that is saturated critical. We claim that $G'$ is also snp-minimal. Otherwise, let $y\in Y$ such that $G'-y$ is snp. By the minimality of $|X\cup Y|$, $G'-y$ is supercyclic, which contradicts that $G'$ is not supercyclic. 

Let us summarize some results from~\cite{superneighbor} on critical bigraphs below for future use. 

\begin{lemma}\label{lem-from-KLLZ} If a bigraph $G=(X, Y)$ is critical, then the following holds for each $x\in X$ and cycle $C$ with $V(C)\cap X = X\setminus\{x\}$:
\begin{itemize} 
    \item [{\em (i)}] $|\mathsf{\Lambda}(x) \cap V(C)| \ge 2$ (Lemma $16$ in~\cite{superneighbor}). \label{lem-from-KLLZ-i}
\end{itemize}

Moreover, if $G$ is saturated critical, then the following hold: 

\begin{itemize}
    \item [{\em (ii)}] If there exists $x_0 \in X$ with degree $2$, then $
    \mathsf{\Lambda}(y) = X$ for every $y \in \mathsf{\Lambda}(x_0)$, and $\deg(x) \ge 4$ for every $x \in X\setminus\{x_0\}$. Consequently, at most one vertex in $X$ can have degree $2$ (Lemma $21$ in \cite{superneighbor}); \label{lem-from-KLLZ-ii}
    
    \item [{\em (iii)}] $\deg(y) \neq |X| - 1$ or $|X| - 2$ for any $y\in Y$ (Lemmas $20$ and $22$ in \cite{superneighbor}). \label{lem-from-KLLZ-iii}
\end{itemize}

Additionally, if $G$ is both saturated critical and snp-minimal, then the following hold:
\begin{itemize}
\item [{\em (iv)}] For any two distinct vertices $y_1, y_2 \in Y$ of degree $2$, $\mathsf{\Lambda}(y_1) \ne \mathsf{\Lambda}(y_2)$ (Lemma $23$ in \cite{superneighbor}); \label{lem-from-KLLZ-iv} 
\item [{\em (v)}] The vertex $x$ has at least two non-neighbors in $V(C)\cap Y$ (Lemma $24$ in~\cite{superneighbor}).\label{lem-from-KLLZ-v}
\end{itemize}
\end{lemma}

Based on Lemma~\ref{lem-from-KLLZ-ii} (ii), we first show that a saturated critical bigraph minimizing $|X|$ has a stronger property:

\begin{lemma}\label{lemma:no-degree-2}
If $G=(X, Y)$ is a saturated critical bigraph such that $|X|$ is minimum among all critical bigraphs, then $\deg(x) \ge 3$ for every $x \in X$. Additionally, if $G$ is also snp-minimal, then there is a vertex $x \in X$ such that $\deg(x) \ge 4$.  
\end{lemma}
\begin{proof}
By Theorem~\ref{thm:KLLZ21-a}, we may assume $|X| \ge 7$. 
We first show that $\deg_G(x) \ge 3$ for every $x \in X$; otherwise, let $x \in X$ such that $\mathsf{\Lambda}_G(x) = \{y, y'\}$. By Lemma~\ref{lem-from-KLLZ-ii} (ii), $\mathsf{\Lambda}_G(y) = \mathsf{\Lambda}_G(y') = X$. We claim that $G' = G-\{x, y\}$ is snp but not supercyclic, which gives a contradiction to the minimality of $|X|$. 

To show $G'$ is snp, we prove $G'$ is dHp instead. For any $A'\subseteq X'$ with $|A'| \ge 2$, since $G$ is snp, it follows that $|\mathsf{\Lambda}_G^2(A'\cup \{x\})| \ge |A'|+1$. As $\mathsf{\Lambda}_G(x) =\{y, y'\}$ and $\mathsf{\Lambda}_G(y') = X$, we have $\mathsf{\Lambda}_{G'}^2(A') = \mathsf{\Lambda}_G^2(A'\cup\{x\})-\{y\}$ and so $|\mathsf{\Lambda}^2_{G'}(A')| \ge |A'|$.
	
To show $G'$ is not supercyclic, it suffices to show that no cycle in $G'$ covers all of $X' = X \setminus \{x\}$. For otherwise, let $C$ be such a cycle. If $y' \notin V(C)$, we replace a segment $x_1y_1x_2 \subset C$ with $x_1yxy'x_2$ to form a cycle in $G$ covering all of $X$, contradicting the criticality of $G$.  If $y' \in V(C)$, assume $y'x'\in E(C)$. By replacing $y'x'$ with $y'xyx'$, we again achieve a contradiction. Thus $G'$ is not supercyclic. 

Therefore $\deg_G(x)\geq 3$ for all $x\in X$.

For the second half of the proof, suppose by contradiction that $\deg_G(x) = 3$ for every $x \in X$. Let $x_0\in X$, and define $X_1=\{x\in X \, : \, |\mathsf{\Lambda}_G(x)\cap \mathsf{\Lambda}_G(x_0)| = 1\}$ and $X_2 = X \setminus (X_1\cup \{x_0\})$. 

Note that for any two distinct vertices $x_1, x_2\in X\setminus\{x_0\}$, there is a cycle $C$ such that $V(C)\cap X=\{x_0, x_1, x_2\}$, which implies $|(\mathsf{\Lambda}_G(x_1)\cup \mathsf{\Lambda}_G(x_2))\cap \mathsf{\Lambda}_G(x_0)| \ge 2$. Since $\deg_G(x_0)=3$, by the Pigeonhole Principle, $|X_1| \le |\mathsf{\Lambda}_G(x_0)| = 3$, and $|X_2| \ge |X| - 4 > 0$. We claim that $X_1 \ne \emptyset$. Otherwise, $|\mathsf{\Lambda}^2_G(X)| \le |\mathsf{\Lambda}_G(x_0)| + \lfloor \frac {|X|-1}{2} \rfloor$ since every vertex in $\mathsf{\Lambda}^2_G(X)\setminus \mathsf{\Lambda}_G(x_0)$ has at least two neighbors in $X \setminus \{x_0\}$. Solving the inequality $3 +\lfloor \frac {|X|-1} {2}\rfloor \ge |X|$, we get $|X| \le 5$, a contradiction. 

Let $x_1 \in X_1$. For any vertex $x_2 \in X \setminus \{x_0, x_1\}$, since $G$ contains a cycle $C$ with $V(C)\cap X=\{x_0, x_1, x_2\}$ and $|\mathsf{\Lambda}_G(x_0)\cap \mathsf{\Lambda}_G(x_1)| = 1$, it follows that $\mathsf{\Lambda}_G(x_2)\cap (\mathsf{\Lambda}_G(x_1)\backslash \mathsf{\Lambda}_G(x_0)) \ne \emptyset$. It also implies that $\mathsf{\Lambda}_G(X_2)\subseteq \mathsf{\Lambda}_G(x_0)\cup \mathsf{\Lambda}_G(x_1)$, and thus $\mathsf{\Lambda}_G(X_2\cup \{x_0, x_1\}) = \mathsf{\Lambda}_G(x_0)\cup \mathsf{\Lambda}_G(x_1)$. Since $ |\mathsf{\Lambda}_G(X_2\cup \{x_0, x_1\})| \ge |\mathsf{\Lambda}_G^2(X_2\cup \{x_0, x_1\})| \ge |X_2| + 2$ and $|\mathsf{\Lambda}_G(x_0)\cup \mathsf{\Lambda}_G(x_1)| = 5$, we know $|X_2| \le 3$. Consequently, the only possible case is $|X| = 7$, $|X_1| = 3$, and $|X_2| = 3$. However, since $G$ contains a cycle $C$ with $V(C)\cap X=X_1$, two vertices in $X_1\backslash\{x_1\}$ must share a common neighbor outside of $\mathsf{\Lambda}_G(x_0)\cup \mathsf{\Lambda}_G(x_1)$. Hence $|\mathsf{\Lambda}^2_G(X)| \le |\mathsf{\Lambda}_G(X)| \le |\mathsf{\Lambda}_G(x_0)\cup \mathsf{\Lambda}_G(x_1)| + 1 = 6 < |X|$. This contradicts the fact that $G$ is snp.
\end{proof}

In the proof of Theorem~\ref{thm:cycle-for-7}, we will frequently use the following observations.

\begin{fact}\label{fact:6_cycle}
Let $G=(X, Y)$ be a critical bigraph. For any three vertices $x_0, x_1, x_2\in X$, there is a cycle $C$ such that $V(C)\cap X=\{x_0, x_1, x_2\}$. Therefore, 
\[
\mathsf{\Lambda}(x_0)\cap \mathsf{\Lambda}(x_1) \ne \emptyset \ne \mathsf{\Lambda}(x_0)\cap \mathsf{\Lambda}(x_2)\ \mbox{ and } |(\mathsf{\Lambda}(x_2) \cup \mathsf{\Lambda}(x_1)) \cap \mathsf{\Lambda}(x_0)| \ge 2.
\]
\end{fact}

For convenience, we assign a clockwise orientation to a cycle $C$ in a graph $G$. For each vertex $a \in V(C)$, we denote its successor by $a^+$ and its predecessor by $a^-$ along the cycle. Similarly, we denote the second successor by $a^{++}$ and the second predecessor by $a^{--}$. For any subset $A \subseteq V(C)$, set $A^+ =\{a^+\, : \, a\in A\}$ and $A^{++} = \{a^{++}\, : \, a\in A\}$; the sets $A^-$ and $A^{--}$ are defined analogously. Given two distinct vertices $a, b \in V(C)$, an {\em $(a, b)$-bridge} in $G$ is defined as either a chord $e\notin E(C)$ together with its endpoints $a,b$ or a non-trivial $a-b$ path $P$ such that $V(P)\cap V(C)=\{a,b\}$.

\begin{lemma}\label{lem-longcycle}
Let $C$ be a longest cycle in a graph $G$ with $V(C) \ne V(G)$, and let $D$ be a component of $G\setminus V(C)$. Then for any two distinct vertices $x, y\in (\mathsf{\Lambda}_C(D))^+$ (or, respectively, $x, y \in (\mathsf{\Lambda}_C(D))^-$), they are not consecutive on $C$, and no $(x, y)$-bridge exists in $G\setminus V(D)$. Consequently, $(\mathsf{\Lambda}_C(D))^+$ and $(\mathsf{\Lambda}_C(D))^-$ are independent sets. 
Moreover, if $w_1$ and $w_2$ are two distinct vertices that form a $w_1-w_2$ path $P$ in $D$, and $x \in (\mathsf{\Lambda}_C(w_1))^{++}, y\in (\mathsf{\Lambda}_C(w_2))^{+}$ (or, respectively, $x \in (\mathsf{\Lambda}_C(w_1))^{--}, y\in (\mathsf{\Lambda}_C(w_2))^{-}$), then $x \ne y$, and no $(x, y)$-bridge exists in $G\setminus V(P)$. In addition, $x$ and $y$ can be consecutive on $C$ if and only if $x^{--} = y^-$ (or, respectively, $x^{++} = y^+$).
\end{lemma}

\subsection {Proof of Theorem~\ref{thm:cycle-for-7}}
Assume to the contrary that $G=(X,Y)$ is an snp bigraph with $|X| = 7$ that is not supercyclic. By Theorem~\ref{thm:KLLZ21-a}, we may assume that $G$ is saturated critical and snp-minimal. 
Let $X = \{x_1, \dots, x_7\}$, where $x_7$ is of maximum degree in $X$. By Lemma~\ref{lemma:no-degree-2}, $\deg (x_7)\geq 4$. Let $C=x_1y_1\cdots x_6y_6x_1$ be a cycle satisfying $V(C)\cap X= X\backslash\{x_7\}$, and let $D$ be the component of $G - V(C)$ containing $x_7$. By Lemma~\ref{lem-from-KLLZ-v} (i) and (v), $|\mathsf{\Lambda}_C(x_7)| \in \{2, 3, 4\}$. Without loss of generality, assume $y_1 \in \mathsf{\Lambda}(x_7)$. We consider three cases according to the size of $\mathsf{\Lambda}_C(x_7)$.

\textbf{Case 1.} $|\mathsf{\Lambda}_C(x_7)| = 2$.
 
Given that $\deg(x_7) \ge 4$ and the snp-minimality of $G$, $V(D) \backslash \{x_7\}$ must include at least two vertices in $Y$, each having at least two neighbors in $X$. Among them, by Lemmas~\ref{lem-from-KLLZ-iv} (iv) and~\ref{lem-longcycle}, all but at most two vertices, say $y,y'\in V(D)$, have at least two neighbors in $X\cap V(C)$, while for these two vertices, we can only guarantee that $\mathsf{\Lambda}_C(y) \neq \emptyset \ne \mathsf{\Lambda}_C(y')$, and $|(\mathsf{\Lambda}_C(y) \cup \mathsf{\Lambda}_C(y'))\cap X| \ge 2$. By symmetry, it suffices to consider when $\mathsf{\Lambda}_C(x_7) =\{y_1, y_2\}, \{y_1, y_3\}$, or $\{y_1, y_4\}$. If $\mathsf{\Lambda}_C(x_7) = \{y_1, y_3 \}$, then at most one vertex in $V(D)\setminus \{x_7\}$ can satisfy the above properties; otherwise, Lemma~\ref{lem-longcycle} would guarantee a cycle covering $X$.

\textbf{Case 1a.} \( \mathsf{\Lambda}_C(x_7) = \{ y_1, y_2 \}\).

By Lemma~\ref{lem-longcycle}, all vertices in $V(D)\setminus\{x_7\}$ can only be adjacent to $x_4, x_5, x_6$ in $C$. Note that $x_5\notin \mathsf{\Lambda}(V(D)\setminus\{x_7\})$, since otherwise, Lemma~\ref{lem-longcycle} would imply that all vertices in $V(D)\setminus \{x_7\}$ are adjacent only to $x_5$ in $C$, leading to a contradiction. Without loss of generality, assume $x_4 \in \mathsf{\Lambda}(y)$, $x_6 \in \mathsf{\Lambda}(y')$, and $\mathsf{\Lambda}(y^*)\supseteq \{x_4, x_6\}$ for any $y^* \in V(D)\setminus\{x_7, y, y'\}$. 
Since $x_2\in (\mathsf{\Lambda}_C(x_7))^+\cap (\mathsf{\Lambda}_C(x_7))^-$, we have $\mathsf{\Lambda}(x_2)\cap (\{y_5, y_6\}\cup V(D))= \emptyset$. Similarly, $\mathsf{\Lambda}(x_3)\cap (\{y_6\}\cup V(D)) = \emptyset$. Also, as $y_5\in (\mathsf{\Lambda}_C(y'))^-$ and $x_3\in (\mathsf{\Lambda}_C(y))^{--}$, $y_5\notin \mathsf{\Lambda}(x_3)$. Thus $(\mathsf{\Lambda}(x_2)\cup \mathsf{\Lambda}(x_3)) \cap (\{y_5, y_6\} \cup V(D)) = \emptyset$. Then from $|(\mathsf{\Lambda}(x_2) \cup \mathsf{\Lambda}(x_3)) \cap \mathsf{\Lambda}(x_6)| \ge 2$, we have $\deg(x_6) > \deg(x_7)$, contradicting the choice of $x_7$. 

\textbf{Case 1b.} $\mathsf{\Lambda}_C(x_7)=\{y_1, y_4\}$.

All vertices in $V(D)\setminus\{x_7\}$ can only be adjacent to $x_3, x_6$ in $C$.
Without loss of generality, we may assume $x_3 \in \mathsf{\Lambda}(y)$, $x_6 \in \mathsf{\Lambda}(y')$ and $\{x_3, x_6\} \subseteq \mathsf{\Lambda}(y^*)$ for any $y^* \in V(D)\setminus\{x_7, y, y'\}$. 
Since $x_2 \in (\mathsf{\Lambda}_C(x_7))^+ \cap (\mathsf{\Lambda}_C(y))^{--}$ and $x_4 \in (\mathsf{\Lambda}_C(x_7))^- \cap (\mathsf{\Lambda}_C(y))^{++}$, we have $(\mathsf{\Lambda}(x_2) \cup \mathsf{\Lambda}(x_4))\cap (\{y_5, y_6\}\cup V(D)) = \emptyset$. Again, because $|(\mathsf{\Lambda}(x_2) \cup \mathsf{\Lambda}(x_4)) \cap \mathsf{\Lambda}(x_6)| \ge 2$, we must have $\deg(x_6) > \deg(x_7)$, contradicting the choice of $x_7$.
            
\textbf{Case 2.} $|\mathsf{\Lambda}_C(x_7)| = 3$.

By symmetry, it suffices to consider the cases when $\mathsf{\Lambda}_C(x_7) = \{y_1, y_2, y_3\}, \{y_1, y_2, y_4\}$, or $\{y_1, y_3, y_5\}$. The snp-minimality of $G$, together with Lemmas~\ref{lem-from-KLLZ-iv} (iv) and~\ref{lem-longcycle}, ensures that there is at most one vertex $y$ in $V(D)\setminus \{x_7\}$, and if such a vertex exists, then $\mathsf{\Lambda}_C(y) \ne \emptyset$. Indeed, if there were two distinct vertices $y, y' \in V(D)\setminus \{x_7\}$, the snp-minimality and Lemma~\ref{lem-from-KLLZ-iv} (iv) would guarantee that $\mathsf{\Lambda}_C(y) \ne \emptyset \ne \mathsf{\Lambda}_C(y')$ and $|\mathsf{\Lambda}_C(y) \cup \mathsf{\Lambda}_C(y')| \ge 2$. Also, in the latter two cases, at least five vertices of $X \cap V(C)$ must lie in $(\mathsf{\Lambda}_C(x_7))^+ \cup (\mathsf{\Lambda}_C(x_7))^-$. However, since $|V(C)\cap X|=6<2+5$, it follows that $(\mathsf{\Lambda}_C(y) \cup \mathsf{\Lambda}_C(y')) \cap ((\mathsf{\Lambda}_C(x_7))^+ \cup (\mathsf{\Lambda}_C(x_7))^-) \ne \emptyset$, which violates Lemma~\ref{lem-longcycle}. In the remaining case when $\mathsf{\Lambda}_C(x_7) = \{y_1, y_2, y_3\}$, we must have $\mathsf{\Lambda}_C(y) \cup \mathsf{\Lambda}_C(y') =\{x_5, x_6\}$, which again contradicts Lemma~\ref{lem-longcycle}. 

Furthermore, note that if $\mathsf{\Lambda}_C(x_7) = \{y_1, y_3, y_5\}$, no such vertex $y$ can exist, as $y$ would otherwise have a neighbor in $V(C)$, leading to a cycle covering $X$. Thus $\deg(x_7) = 3$, contradicting the choice of $x_7$. Therefore, it remains to consider the following two cases.

\textbf{Case 2a.} $\mathsf{\Lambda}_C(x_7) = \{y_1, y_2, y_3\}$.
 
Given that $\deg(x_7) \ge 4$, there is exactly one vertex $y$ in $V(D)\setminus\{x_7\}$, and by Lemma~\ref{lem-longcycle}, $y$ can only be adjacent to $x_5, x_6$ in $C$. Assume $x_5 \in \mathsf{\Lambda}_C(y)$. Since $x_2, x_3 \in (\mathsf{\Lambda}_C(x_7))^+ \cap (\mathsf{\Lambda}_C(x_7))^-$, we have $(\mathsf{\Lambda}(x_2) \cup \mathsf{\Lambda}(x_3)) \cap \{y_4, y_5, y\} = \emptyset$. However, since $|(\mathsf{\Lambda}(x_2) \cup \mathsf{\Lambda}(x_3)) \cap \mathsf{\Lambda}(x_5)| \ge 2$, we must have $\deg(x_5) \ge 5 > \deg(x_7)$, contradicting the choice of $x_7$. So $x_5 \notin \mathsf{\Lambda}_C(y)$. Similarly, $x_6 \notin \mathsf{\Lambda}_C(y)$, and so $\mathsf{\Lambda}_C(y) = \emptyset$, a contradiction.
        
\textbf{Case 2b.} \(\mathsf{\Lambda}_C(x_7) = \{y_1, y_2, y_4\}\).

In this case, the unique vertex $y \in V(D)\setminus\{x_7\}$ satisfies $\mathsf{\Lambda}_C(y) = \{x_6\}$. 
Since $x_2 \in (\mathsf{\Lambda}_C(x_7))^+ \cap (\mathsf{\Lambda}_C(x_7))^-$, we have $\mathsf{\Lambda}(x_2) \cap \{y, y_5, y_6\} = \emptyset$. Also, $x_3 \in (\mathsf{\Lambda}_C(x_7))^+$ and $x_4 \in (\mathsf{\Lambda}_C(x_7))^-$, implying that $\mathsf{\Lambda}(x_3) \cap \{y, y_6\} = \mathsf{\Lambda}(x_4) \cap \{y, y_5\} = \emptyset$. Note that $\deg(x_6) \le \deg(x_7) = 4$. Then by Fact~\ref{fact:6_cycle} applied to $\{x_2, x_3, x_6\}$ and to $\{x_2, x_4, x_6\}$, $x_2$ and $x_6$ share a unique neighbor outside of $\{y, y_5, y_6\}$, and $x_3y_5, x_4y_6 \in E(G)$. 
We now consider each possible vertex in $\mathsf{\Lambda}(x_2) \cap \mathsf{\Lambda}(x_6)$ and show that, in each scenario, a contradiction arises, thereby finishing this case.

If $\mathsf{\Lambda}(x_2) \cap \mathsf{\Lambda}(x_6) = \{y_2\},\{y_3\}$, or $\{y'\}$ where $y' \notin V(C)\cup \{y\}$, then a $14$-cycle must exist in $G$ and give a contradiction: $x_7yx_6y_2x_2y_1x_1y_6x_4y_3x_3y_5x_5y_4x_7$, $x_7y_2x_3y_3x_2y_1x_1y_6x_4y_4x_5y_5x_6yx_7$, or $x_7yx_6y'x_2y_1x_1y_6x_4y_4x_5y_5x_3y_2x_7$ respectively.
If $\mathsf{\Lambda}(x_2) \cap \mathsf{\Lambda}(x_6) = \{y_4\}$, then $\deg(y_4) \ge 5$, which implies $\deg(y_4) = |X| = 7$ by Lemma~\ref{lem-from-KLLZ-iii} (iii). Then $G$ has a $14$-cycle $x_7yx_6y_6x_4y_3x_3y_5x_5y_4x_1y_1x_2y_2x_7$, a contradiction.
If $\mathsf{\Lambda}(x_2) \cap \mathsf{\Lambda}(x_6) = \{y_1\}$, then $\mathsf{\Lambda}(x_6) = \{y_1, y_5, y_6, y\}$. Note that $x_5 \in (\mathsf{\Lambda}_C(x_7))^+$ and thus $\mathsf{\Lambda}(x_5) \cap \{y, y_6\} = \emptyset$. Since $|(\mathsf{\Lambda}(x_3) \cup \mathsf{\Lambda}(x_5)) \cap \mathsf{\Lambda}(x_6)| \ge 2$, we have $y_1 \in \mathsf{\Lambda}(x_3)\cup \mathsf{\Lambda}(x_5)$ which means $\deg(y_1) \ge 5$ and so $\deg(y_1) = 7$. Additionally, $|(\mathsf{\Lambda}(x_1) \cup \mathsf{\Lambda}(x_6)) \cap \mathsf{\Lambda}(x_2)| \ge 2$, implying the existence of a vertex $y' \in Y\setminus\{y_1, y, y_5, y_6\}$ such that $y' \in \mathsf{\Lambda}(x_1) \cap \mathsf{\Lambda}(x_2)$. If $y' \notin V(C) \cup \{y\}$, then $y'$ can be substituted for $y_1$ in $V(C)$, making $y_1$ a vertex in $V(D)$ that is adjacent to all vertices in $X$, which contradicts Lemma~\ref{lem-longcycle}. Furthermore, if $y' = y_2$, then a $14$-cycle exists in $G$:  $x_7yx_6y_5x_3y_3x_4y_6x_1y_2x_2y_1x_5y_4x_7$. Similarly, if $y' = y_3$, there exists another $14$-cycle in $G$: $x_7yx_6y_5x_3y_2x_2y_3x_4y_6x_1y_1x_5y_4x_7$. Finally, if $y' = y_4$, then $\deg(y_4) \ge 5$, and thus $\deg(y_4) = 7$, which implies that $\deg(x_6) > \deg(x_7)$, leading to a contradiction again. Thus no such $y'$ can exist, a contradiction.

For the remainder of the proof, we will frequently rely on the following observation. 
\begin{lemma}\label{lem-bridge-pro}
Let $C$ be the longest cycle in a graph $G$ with $V(C) \ne V(G)$ and a fixed clockwise orientation. Given $x, y, z \in V(C)$, let $P_1$ and $P_2$ be two internally disjoint bridges, specifically an $(x, y)$-bridge and an $(x^+, z)$-bridge, with at least one of these bridges having length at least $2$. Assume that $z \notin \{x, x^+, x^{++}\}$. Let $C^+[x^{++}, z^-]$ denote the consecutive vertices on $C$ from $x^{++}$ to $z^-$ in the clockwise direction. If $y \in C^+[x^{++}, z^-]$, then there exists no $(y^-, z^+)$-bridge, $(y^-, z^-)$-bridge, or $(y^+, z^+)$-bridge that is internally disjoint from both $P_1$ and $P_2$. In addition, if $y = z$, then no $(y^-, z^+)$-bridge exists that is internally disjoint from $P_1$ and $P_2$. 
\end{lemma}

\textbf{Case 3.} $|\mathsf{\Lambda}_C(x_7)| = 4$.

By symmetry, we may assume that $\mathsf{\Lambda}_C(x_7) = \{y_1, y_2, y_3, y_4\}$, $\{y_1, y_2, y_3, y_5\}$, or $\{y_1, y_2, y_4, y_5\}$. Also, to satisfy $|\mathsf{\Lambda}^2(X)| \ge |X|$, there must be a vertex $y \in Y\backslash
V(C)$ such that $\deg (y)\geq 2$.

\textbf{Case 3a.} \(\mathsf{\Lambda}_C(x_7) = \{y_1, y_2, y_3, y_4\}\).
            
By Lemma~\ref{lem-longcycle}, $y$ must satisfy one of the following configurations:
$\mathsf{\Lambda}(y) = \{x_1, x_5, x_6\},\{x_1, x_5\}$ or $\{x_i, x_6\}$ where $i\neq 6$.
Notice that if $\mathsf{\Lambda}(y) = \{x_1, x_6\}$, then by Lemma~\ref{lem-from-KLLZ-iv} (iv), at least one vertex in $X\setminus \{x_1, x_6\}$ must be adjacent to $y_6$. This configuration is effectively equivalent to $\mathsf{\Lambda}(y) = \{x_1, x_5, x_6\}$, as $y$ can be substituted into the cycle $C$ for $y_6$, and then $y_6$ becomes the vertex outside $C$ with $\{x_1, x_6\} \subset \mathsf{\Lambda}(y_6)$. Similarly, the configuration $\mathsf{\Lambda}(y) = \{x_5,x_6\}$ is also equivalent to $\mathsf{\Lambda}(y) = \{x_1, x_5, x_6\}$. Additionally, if $\mathsf{\Lambda}(y) = \{x_6, x_7\}$, then $\deg(y_5) = \deg(y_6) = \deg(y) = 2$ by Lemma~\ref{lem-longcycle}. Since $|\mathsf{\Lambda}^2(X\setminus\{x_5, x_6\})| \ge |X\setminus\{x_5, x_6\}| = 5$, there exists at least one vertex $y' \notin V(C)\cup \{y\}$ such that $y'$ is adjacent to at least two vertices in $X\setminus\{x_5, x_6\}$, as $y_5, y_6, y \notin \mathsf{\Lambda}^2(X\setminus\{x_5, x_6\})$. By Lemma~\ref{lem-longcycle}, we know such a $y'$ does not exist, implying that $\mathsf{\Lambda}(y) \ne \{x_6, x_7\}$. It also implies that $\mathsf{\Lambda}(x_7)= \mathsf{\Lambda}_C(x_7)= \{y_1, y_2, y_3, y_4\}$. Now we consider the remaining possibilities for $y$: either $\mathsf{\Lambda}(y) = \{x_i, x_6\}$ for some $x_i \in \{x_2, x_3, x_4\}$, $\mathsf{\Lambda}(y) = \{x_1, x_5\}$, or $\mathsf{\Lambda}(y) = \{x_1, x_5, x_6\}$. In each scenario, we show that a contradiction arises, thereby concluding this case.

Suppose $\mathsf{\Lambda}(y) = \{x_1, x_5\}$. Since $|\mathsf{\Lambda}^2(X\setminus \{x_1, x_5\})| \ge |X\setminus \{x_1, x_5\}| = 5$, there is a vertex $y'\in (Y\setminus(V(C)\cup\{y\})) \cup \{y_5, y_6\}$ such that $y'$ is adjacent to at least two vertices in $X\setminus\{x_1, x_5, x_7\}$. If $y' \notin V(C)\cup\{y\}$, then $y'$ must satisfy $\mathsf{\Lambda}(y') =\{x_i, x_6\}$ for some $x_i \in \{x_2, x_3, x_4\}$. However, in all cases, we get a $14$-cycle: $x_7y_1x_2y'x_6y_6x_1yx_5y_4x_4y_3x_3y_2x_7$, $x_7y_2x_2y_1x_1yx_5y_5x_6y'x_3y_3x_4y_4x_7$, or $x_7y_4x_4y'x_6y_5x_5yx_1y_1x_2y_2x_3y_3x_7$, respectively. If $y' = y_5$, then $|\mathsf{\Lambda}(y_5) \cap \{x_2, x_3, x_4\}| \ge 1$, contradicting Lemma~\ref{lem-bridge-pro}. To see this, assume $y_5x_i \in E(G)$ for some $i \in \{2, 3, 4\}$. Then the vertices $x_i, y_5$ and $y_1$ correspond to $x, y$, and $z$ in Lemma~\ref{lem-bridge-pro}, respectively. This configuration satisfies the conditions of the lemma and leads to a contradiction, as a $(y^-, z^-)$-bridge (the $(x_5, x_1)$-bridge) exists. Similarly, if $y' = y_6$, then $|\mathsf{\Lambda}(y_6) \cap \{x_2, x_3, x_4\}| \ge 1$, again contradicting Lemma~\ref{lem-bridge-pro}. Assume $y_6x_i \in E(G)$ for some $i \in \{2, 3, 4\}$. Here, the vertices $y_{i-1}, y_4$ and $y_6$ correspond to $x, y$, and $z$ in Lemma~\ref{lem-bridge-pro}, respectively. This configuration meets the conditions of the lemma, and the existence of a $(y^+, z^+)$-bridge (the $(x_5, x_1)$-bridge) yields a contradiction. 
Thus such a $y'$ does not exist, a contradiction. By a similar argument, we can show that $\mathsf{\Lambda}(y) \ne \{x_1, x_5, x_6\}$. 

Suppose there is $x_i \in \{x_2, x_3, x_4\}$ such that $\mathsf{\Lambda}(y) = \{x_i, x_6\}$. Assume $\mathsf{\Lambda}(y) = \{x_2, x_6\}$. Since $|\mathsf{\Lambda}^2(X\setminus\{x_5, x_6\})| \ge |X\setminus\{x_5, x_6\}| = 5$, then there exists a vertex $y' \in (Y \setminus(V(C)\cup \{y\}))\cup \{y_5, y_6\}$ such that $|\mathsf{\Lambda}(y') \cap (X\setminus\{x_5, x_6, x_7\})| \ge 2$. By Lemma~\ref{lem-longcycle}, we know $y' \in V(C)$, meaning that $y' \in \{y_5, y_6\}$. This implies that $|(\mathsf{\Lambda}(y_5) \cup \mathsf{\Lambda}(y_6)) \cap \{x_2, x_3, x_4\}| \ge 1$, contradicting Lemma~\ref{lem-bridge-pro}. Thus such a $y'$ does not exist, a contradiction. We conclude that $\mathsf{\Lambda}(y) \ne \{x_4, x_6\}$ by symmetry, and $\mathsf{\Lambda}(y) \ne \{x_3, x_6\}$ by using a similar argument.
             
\textbf{Case 3b.} \(\mathsf{\Lambda}_C(x_7) = \{y_1, y_2, y_3, y_5\}\).
             
By Lemma~\ref{lem-longcycle}, $x_7$ cannot have any other neighbors, and $y$ must satisfy one of the following:
$\mathsf{\Lambda}(y) =\{x_1, x_4\},\{x_1, x_6\},\{x_4, x_5\},$ or $ \{x_5, x_6\} $. 
Note that if $\mathsf{\Lambda}(y) = \{x_4, x_5\}$, then $\mathsf{\Lambda}(y_4) = \{x_4, x_5\} = \mathsf{\Lambda}(y)$, which contradicts Lemma~\ref{lem-from-KLLZ-iv} (iv). For the same reason, $\mathsf{\Lambda}(y) \ne \{x_1, x_6\}$. Furthermore, if $\mathsf{\Lambda}(y) = \{x_5, x_6\}$, replacing $y_5$ in $C$ with $y$ reduces the scenario to \textbf{Case 2a}. Finally, if $\mathsf{\Lambda}(y) = \{x_1, x_4\}$, then we consider the set $A = X\setminus\{x_1, x_4\}$. Since $|\mathsf{\Lambda}^2(A)| \ge |A| = 5$, there exists a vertex $y' \in (Y\setminus(V(C)\cup\{y\}))\cup \{y_4, y_6\}$ adjacent to at least two vertices in $X\setminus\{x_1, x_4, x_7\}$. From our analysis above, $y'$ must be in $V(C)$, that is, $y' \in \{y_4, y_6\}$. If $y' = y_4$, then $|\mathsf{\Lambda}(y_4) \cap \{x_2, x_3, x_6\}| \ge 1$, and if $y' = y_6$, then $|\mathsf{\Lambda}(y_6) \cap \{x_2, x_3, x_5\}| \ge 1$; either case contradicts Lemma~\ref{lem-bridge-pro}.

\textbf{Case 3c.} \(\mathsf{\Lambda}_C(x_7) = \{y_1, y_2, y_4, y_5\}\).
             
Again, $\mathsf{\Lambda}(x_7) = \{y_1, y_2, y_4, y_5\}$, and $y$ must satisfy one of the following: $\mathsf{\Lambda}(y) = \{x_1, x_3\}, \{x_1, x_6\}$, $\{x_3, x_4\}$, or $\{x_4, x_6\}$. If $\mathsf{\Lambda}(y) = \{x_1, x_6\}$, then $\mathsf{\Lambda}(y_6) = \{x_1, x_6\} = \mathsf{\Lambda}(y)$, contradicting Lemma~\ref{lem-from-KLLZ-iv} (iv). Similarly, $\mathsf{\Lambda}(y) \ne \{x_3, x_4\}$.
By symmetry, it suffices to consider one of the remaining cases.
Assume $\mathsf{\Lambda}(y) = \{x_1, x_3\}$. Let $A = X\setminus \{x_1, x_3\}$. Since $|\mathsf{\Lambda}^2(A)| \ge |A| = 5$, there exists a vertex $y' \in (Y\setminus(V(C)\cup \{y\}))\cup \{y_3, y_6\}$ satisfying $|\mathsf{\Lambda}(y') \cap (X\setminus\{x_1, x_3, x_7\})| \ge 2$. If $y' \notin V(C)\cup\{y\}$, then $y'$ must satisfy $\mathsf{\Lambda}(y') = \{x_4, x_6\}$, and then there exists a $14$-cycle $x_7y_5x_5y_4x_4y'x_6y_6x_1yx_3y_2x_2y_1x_7$, a contradiction. Thus $y' \in V(C)$, that is, $y' \in \{y_3, y_6\}$. If $y' = y_3$, then $|\mathsf{\Lambda}(y_3) \cap \{x_2, x_5, x_6\}| \ge 1$, and if $y' = y_6$, then $|\mathsf{\Lambda}(y_6) \cap \{x_2, x_4, x_5\}| \ge 1$, both of which contradict Lemma~\ref{lem-bridge-pro}.
\qed

\section{Two equivalent conjectures} \label{section:big-degree-Y-for-7}

In this section, we will focus on Conjectures~\ref{conj:Big-degree-Y} and~\ref{subconj:Big-degree-Y} and provide a proof of Theorem~\ref{thm:Big-degree-Y-for-7}. We begin by defining some notation. Let $P$ denote a path in a bipartite graph $G$ with a fixed orientation. For each vertex $u \in V(P)$, let $x^+_P(u)$ (respectively, $x^-_P(u)$) denote the closest vertex in $X\backslash \{u\}$ that is the successor (respectively, predecessor) of $u$ along the path $P$. For a set $U \subseteq V(P)$, define $X_P^+(U) = \{x^+_P(u) : u \in U\}$. The vertices $y^+_P(u)$, $y^-_P(u)$ and the sets $X^-_P(U)$, $Y^+_P(U)$, $Y^-_P(U)$ are defined analogously. When the path $P$ is clear from the content, the subscript may be omitted.

With these definitions, we now introduce two lemmas that will act as the basis of our proof.

\begin{lemma}\label{lemma:y-y'}
If a bigraph $G=(X, Y)$ with $|X| = n$ has a $y-y'$ path covering all vertices in $X$ for some $y, y' \in Y$ with $\deg(y) + \deg(y') \ge n + 2$, then $G$ has a cycle covering all of $X$.
\end{lemma}
\begin{proof}
Let $P = y_1x_1y_2x_2 \dots y_nx_ny_{n+1}$ be such a path, where $y_1 = y$ and $y_{n+1} = y'$. Orient $P$ from $y_1$ to $y_{n+1}$. Notice that since $|X^-(\mathsf{\Lambda}_P(y_1))| = \deg(y_1) - 1$ and $|\mathsf{\Lambda}_P(y_{n+1})| = \deg(y_{n+1})$, we have $|X^-(\mathsf{\Lambda}_P(y_1))| + |\mathsf{\Lambda}_P(y_{n+1})| \ge n + 1$. Hence there is some vertex $u\in X$ in both $X^-(\mathsf{\Lambda}_P(y_1))$ and $ \mathsf{\Lambda}_P(y_{n+1})$, so $G$ has a cycle covering all of $X$ using $V(P) \setminus \{y_P^+(u)\}$.
\end{proof}

Additionally, if every vertex $y \in Y$ in $G$ satisfies $\deg(y) > \frac{n+1}{2}$, then we also have the following:

\begin{lemma}\label{lemma:G + xy}
Suppose there exists a non-adjacent pair $x \in X$, $y\in Y$ such that $\deg(x) + \deg(y) \ge n+1$. Then $G + xy$ has a cycle covering all of $X$ if and only if $G$ has a cycle covering all of $X$.
\end{lemma}
\begin{proof}
Proving the sufficiency is trivial. Now we consider the necessity.

Suppose $G+xy$ has a cycle covering all of $X$. Then $G$ has a $x-y$ path that covers all of $X$. Let $P = yx_1y_2x_2 \dots y_nx$, oriented from $y$ to $x$. If there exists a vertex $y' \in \mathsf{\Lambda}_G(x)\backslash V(P)$, then $G$ contains a $y-y'$ path covering all of $X$. Since $\deg(y) > \frac{n+1}{2}$ for all $y \in Y$, then by Lemma~\ref{lemma:y-y'}, there is a cycle in $G$ covering all of $X$. Now suppose $\mathsf{\Lambda}_G(x) \subseteq V(P)$. Since $|Y^-(\mathsf{\Lambda}_P(y))| = \deg(y)$, it follows that $|Y^-(\mathsf{\Lambda}_P(y))| + |\mathsf{\Lambda}_P(x)| = \deg(y) + \deg(x) \ge n+1$. So $Y^-(\mathsf{\Lambda}_P(y_1))$ and $\mathsf{\Lambda}_P(x_n)$ share a common vertex, and thus $G$ has a cycle covering all of $X$.
\end{proof}

Let $G=(X, Y)$ be a dHp bigraph satisfying $|X| = n > \max \{2k+1, k(k+1)\}$ and $\deg(y) \ge n - k$ for all $y \in Y$. Define $X_s = \{x \in X: \deg(x) \leq k\}$ as the set of vertices in $X$ with small degree, and $X_l = X \setminus X_s$ as the set of vertices in $X$ with large degree. Note that for any $x \in X_l$ and $y \in Y$, we have $\deg(y) \ge n-k > \frac{n+1}{2}$ and $\deg(x) + \deg(y) \geq (k+1)+(n-k) = n+1$. We can derive a new bigraph $H$ from $G$ by adding edges to make $X_l$ and $Y$ form a complete bipartite graph. Since $G$ satisfies the double Hall property, $H$ also satisfies the double Hall property.
 
By directly double-counting the edges in $H$, we have the following claim.

\begin{claim}\label{claim:max-low-degree-count}
     $|X_s| \leq \frac{|Y|k}{|Y|-k}$. In particular, $n > \max\{2k+1, k(k+1)\}$ implies $|X_s| \leq k$, and $|X_l| \ge k + 2$.
\end{claim}
\begin{proof}
As $H$ is a dHp graph, $|Y| \ge |\mathsf{\Lambda}^2_H(X)| \ge |X| = n$. By construction, $H$ also satisfies $\deg_H (y)\geq n-k$ for each $y\in Y$, implying that $H$ contains at least $|Y|(n-k)$ edges. On the other hand, if $s = |X_s|$, then the maximum number of edges in $H$ is $s k + (n-s)|Y|$. Therefore $sk + (n-s)|Y| \ge |Y|(n-k)$ and so $$ s \le \frac{|Y|k}{|Y|-k} = k+\frac{k^2}{|Y|-k}.$$
If $n > \max\{2k+1, k(k+1)\}$, then $\frac{k^2}{|Y|-k} < 1$, and thus $s \le k$. As $n>2k+1$, this implies that $|X_l|\ge k+2$.
\end{proof}

Next we show the equivalence in the graph $H$ between a cycle covering $X$ and a collection of $Y-Y$ paths whose union covers exactly $X_s$ (that is, covers all of $X_s$ but none of the vertices in $X_l$).

\begin{claim}\label{claim:paths-for-Xs} 
$H$ contains a collection of disjoint and nontrivial $Y-Y$ paths $P_1,\dots, P_m$ whose union covers exactly $X_s$ if and only if there is a cycle covering all of $X$ in $H$.  
\end{claim}
\begin{proof}
Proving the sufficiency is trivial. For the necessity, we may assume $|X_s| \ge 1$; if $|X_s| = 0$, then $H$ is a complete bigraph with $|Y| \ge |X|$, which clearly contains a cycle covering $X$.

By Claim~\ref{claim:max-low-degree-count}, we have $|X_l| \ge k + 2 \ge |X_s| + 2 \ge m + 2$. Since all the vertices in $X_l$ have degree $|Y|$, we can join these $m$ paths by using $m-1$ distinct vertices in $X_l$ and form a $Y-Y$ path $P$ that contains all vertices in $X_s$. Let $y,y'\in Y$ be the endpoints of $P$. Also, $|Y\setminus V(P)|\geq |X\setminus V(P)|-1 \ge |X_l \setminus V(P)|-1$,  so there exists an $X-X$ path $P'$ with endpoints $x,x' \in X_l \setminus V(P)$ that can cover the remaining vertices in $X_l\setminus V(P)$. Then the cycle $xP'x'yPy'x$ covers all of $X$ in $H$.
\end{proof}

By the construction of $H$ and Lemma~\ref{lemma:G + xy}, the following result can be directly derived from Claim~\ref{claim:paths-for-Xs}.
\begin{fact}\label{fact:re-paths-for-Xs}
$G$ contains a collection of disjoint and nontrivial $Y-Y$ paths $P_1,\dots, P_m$ whose union covers exactly $X_s$ if and only if $G$ has a cycle covering all of $X$.
\end{fact}

Now we are ready to show the equivalence between the two conjectures.

\begin{lemma} \label{lemma:equivalent-conjecture}
 Conjecture~\ref{conj:Big-degree-Y} is equivalent to Conjecture~\ref{subconj:Big-degree-Y}.
\end{lemma}
\begin{proof}
Suppose Conjecture~\ref{subconj:Big-degree-Y} is true. Let $G$ be a dHp graph which satisfies the conditions in Conjecture~\ref{conj:Big-degree-Y} for some $k$. If $|X_s|= 0$, then by Fact~\ref{fact:re-paths-for-Xs}, we are done. We may assume $|X_s| \ge 1$, and then $|X_l| \ge |X_s| + 2$. Let $x\in X_l$. By checking the double Hall property on $X_s \cup \{x\}$ in $G$, we have $|\mathsf{\Lambda}^2_G(X_s \cup \{x\})| \ge |X_s|+1$, and thus $|\mathsf{\Lambda}_G(X_s)| \ge |X_s|+1$, as $\mathsf{\Lambda}^2_G(X_s \cup \{x\}) \subseteq \mathsf{\Lambda}_G(X_s)$.
So if $|X_s| = 1$, we can find a $Y-Y$ path that covers exactly the only vertex in $X_s$, and by Fact~\ref{fact:re-paths-for-Xs}, $G$ contains a cycle covering all of $X$. Now suppose $|X_s| \ge 2$.
Let $G^*=G[X_s \cup \mathsf{\Lambda}_G(X_s)]$. Clearly, $G^*$ is dHp and $|\mathsf{\Lambda}_{G^*}(X_s)| = |\mathsf{\Lambda}_G(X_s)| \ge |X_s| + 1$, indicating that $G^*$ meets the conditions of Conjecture~\ref{subconj:Big-degree-Y}. Then there exists a collection of disjoint $Y-Y$ paths in $G^*$, and hence in $G$, whose union covers exactly $X_s$. By Fact~\ref{fact:re-paths-for-Xs}, $G$ contains a cycle covering all of $X$. Therefore Conjecture~\ref{conj:Big-degree-Y} holds.
	
Suppose Conjecture~\ref{conj:Big-degree-Y} is true. Let $G=(X,Y)$ be a dHp bigraph such that $|\mathsf{\Lambda}_G(X)| \ge |X| + 1$. Assume $|X|=k$. Let $n>\max \{2k+1,k(k+1)\}$. Construct the bigraph $G^*=(X^*,Y^*)$ by adding $n-k$ vertices to $X$ and $n-k$ vertices to $Y$, and joining every $x\in X^*\backslash X$ to all vertices in $Y^*$. It is easy to see that $G^*$ meets the conditions of Conjecture~\ref{conj:Big-degree-Y}. Then there is a cycle in $G^*$ that covers $X^*$. Removing $X^*\setminus X$ from the cycle gives a collection of disjoint $Y-Y$ paths in $G^*$ whose union covers $X$, and these paths also exist in $G$, since any edge between $X$ and $Y^*$ in $G^*$ also exists in $G$. Therefore we confirm Conjecture~\ref{subconj:Big-degree-Y}.
\end{proof} 

Finally, we give a short proof of Theorem~\ref{thm:Big-degree-Y-for-7}. 
\begin{proof}[Proof of Theorem~\ref{thm:Big-degree-Y-for-7}]
Since $0 \le k \le 7$, by Claim~\ref{claim:max-low-degree-count}, we know $|X_s| \le 7$.
By applying Lemma~\ref{lemma:equivalent-conjecture}, the task is reduced to showing that Conjecture~\ref{subconj:Big-degree-Y} holds for $2 \le |X| \le 7$, as this conjecture is straightforward to verify for $|X| = 0$ or $1$. Observe that, over the range of $2 \le |X| \le 7$, it suffices to show that there is a cycle that covers $X$. Hence we can say that Conjecture~\ref{con-Salia} implies Conjecture~\ref{subconj:Big-degree-Y}. By Theorem~\ref{thm:cycle-for-7}, Conjecture~\ref{con-Salia} is valid for $2 \le |X| \le 7$, as it combines the double Hall property for $|X| = 2$ with the fact that every dHp bigraph with $|X| \ge 3$ is snp. Consequently, Conjecture~\ref{subconj:Big-degree-Y} is confirmed for this range, completing the proof.
\end{proof}

Conversely, however, we can only state that Conjecture~\ref{subconj:Big-degree-Y} nearly implies Conjecture~\ref{con-Salia}, as the former considers the ``collection of disjoint $Y-Y$ paths,'' while the latter focuses on a cycle. But in the important special case where $|\mathsf{\Lambda}(X)|$ is very close to $|X|$, the collection of disjoint $Y-Y$ paths covering $X$ is nearly equivalent to a cycle covering $X$.  

\section{Proof of Theorem~\ref{thm:set-of-degree-Y}} \label{section:set-of-degree-Y}
In this section, we will apply Hall's Theorem~\cite{diestel} to prove Theorem~\ref{thm:set-of-degree-Y}. The idea is first to use only vertices in $Y$ with degree $2$, along with all vertices in $X$, to form disjoint $X-X$ paths. Then, we find enough suitable vertices in $Y$ with large degree to join the paths.
\begin{theorem}\textnormal{(Hall's Theorem)}\label{hall-thm}
A bipartite graph $G=(X, Y)$ contains a matching covering all vertices of $X$ if and only if $|\mathsf{\Lambda}(S)| \geq |S|$ for all $S \subseteq X$.
\end{theorem}

\noindent \textit{Proof of Theorem~\ref{thm:set-of-degree-Y}.} Define $Y_s = \{y \in Y: \deg(y) = 2\}$ and $Y_l = Y\setminus Y_s$. Among all collections of disjoint $X-X$ paths whose union covers $X$ using only vertices in $Y_s\cup X$, let $P_1, P_2, P_3,\dots, P_m$ be a collection in which $m$ is minimized. The paths have endpoints $(x_L^1, x_R^1)$ through $(x_L^m, x_R^m)$. Define $L = \{x_L^1, \dots, x_L^m\}$ and $R = \{x_R^1, \dots, x_R^m\}$.  
Note that it is possible to have $x_L^i = x_R^i$ for some $i$.

When $m = 1$, since $|X| \ge 2$, the path $P_1$ is an $X-X$ path with distinct endpoints that covers all vertices in $X$. Applying the double Hall property to these two endpoints yields at least two common neighbors in $Y\setminus V(P_1)$. Selecting one of them allows us to join the endpoints of $P_1$, thereby forming the desired cycle.

Now we assume $m \ge 2$. Note that $\mathsf{\Lambda}^2(L),\mathsf{\Lambda}^2(R) \subseteq Y_l$; otherwise, we could have made the size of the collection of $X-X$ paths smaller in the first step. By applying the double Hall property to $R$, where $|R| = m \ge 2$, we have $|Y_l| \ge m$. Next, we show that the paths can be joined together using the high-degree vertices in $Y_l$. To do this, consider the \emph{intended pairs} $(x_R^i, x_L^{i+1})$ for $i=1, \dots, m$ (with $x_L^{m+1} = x_L^1$). Our goal is to find a high-degree common neighbor for each intended pair.

Let $F$ be the bipartite graph with the set of the intended pairs as vertices in one partition and $Y_l$ as the other. An intended pair $(x_R^i, x_L^{i+1})$ is adjacent to $y \in Y_l$ in $F$ if $y$ is a common neighbor of $x_R^i$ and $x_L^{i+1}$. Furthermore, assume these disjoint paths are ordered in such a way that $|E(F)|$ is maximized. If there exists a matching $M$ in $F$ that covers all the intended pairs, then the corresponding edges in $G$ together with $m$ paths form a cycle covering all of $X$. Therefore it suffices to verify Hall's condition in $F$.
Let $S$ be a set of intended pairs in $F$, where $1 \le |S| \le m$, and let $V(S) \subseteq X$ denote the set of all vertices that appear in any pair in $S$.

If $|S| \le 2$, then the double Hall property implies that the two vertices in each intended pair in $S$ have at least two common neighbors. These common neighbors must be in $Y_l$ by the minimality of $m$. Thus $|\mathsf{\Lambda}_F(S)| \ge 2 \ge |S|$.

If $|S| \geq 5$, we claim $|\mathsf{\Lambda}_F(S)| \ge m \ge |S|$. Since any vertex $y \in Y_l$ satisfies $\deg_G(y)\ge |X|-2$, at most two vertices in $X$ are not adjacent to $y$. Then $y$ has at most four non-neighbors in $F$. However, if $|S| \geq 5$, every $y \in Y_l$ has at least one neighbor in $S$. Thus $|\mathsf{\Lambda}_F(S)| = |Y_l| \ge m$. 

The remaining cases are $|S| = 3$ and $|S| = 4$. Note that if each $P_i$ has length at least $2$, then every $y\in Y_l$ has at least one neighbor in $S$, so $|\mathsf{\Lambda}_F(S)|=|Y_l|$. Thus a potential obstacle arises only if there is some path consisting of a single vertex and some $y \in Y_l$ is non-adjacent to this vertex. Note again that the common neighbor of any pair of vertices that are endpoints of different paths must in $Y_l$.

Now, assume for contradiction that some set $S$ fails Hall's condition in $F$.

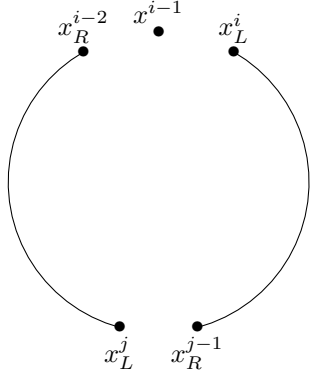
\begin{figure}[H]
\centering
\begin{tikzpicture}[scale=0.9]
\node at (120:2) {$\bullet$};
\node at (90:2) {$\bullet$};
\node at (60:2) {$\bullet$};
\node at (-75:2) {$\bullet$};
\node at (-105:2) {$\bullet$};
\draw [domain=-75:60] plot ({2*cos(\x)}, {2*sin(\x)});
\draw [domain=120:255] plot ({2*cos(\x)}, {2*sin(\x)});
\node [above] at (120:2) {$x_R^{i-2}$};
\node [above] at (90:2) {$x^{i-1}$};
\node [above] at (60:2) {$x_L^i$};
\node [below] at (-75:2) {$x_R^{j-1}$};
\node [below] at (-105:2) {$x_L^j$};
\end{tikzpicture}
\caption{Example for $|S| = 3$}\label{example_1}
\end{figure}

Suppose $|S| = 3$. Figure 1 shows the three intended pairs $(x_R^{i-2},x^{i-1}),(x^{i-1},x_L^i),(x_R^{j-1},x_L^j)$ in $S$, where $P_{i-1}=x^{i-1}$. If both segments $x_L^j-x_R^{i-2}$ and $x_R^{j-1}-x_L^i$ are trivial, i.e., an isolated vertex, then $|X| = 3$, and by Theorem~\ref{thm:cycle-for-7}, $G$ has a cycle covering all of $X$. Thus we may assume $x_L^j\neq x_R^{i-2}$. Since $G$ is a dHp graph, the only case where $|\mathsf{\Lambda}_F(S)| < |S| = 3$ arises when exactly two vertices $y, y'\in Y_l$ are adjacent to every vertex in $V(S)$, thereby securing the double Hall property for each pair, while every other vertex in $Y_l$ is not adjacent to any pair in $S$. Now checking the double Hall property on $L \cap V(S)$ implies the existence of  $y'' \in Y_l$ that is adjacent to exactly two vertices in $L \cap V(S)$. Note that at least one of these vertices must belong to a non-trivial segment; otherwise, they would be an intended pair in $S$, with $y''$ as its neighbor in $F$, contradicting $|\mathsf{\Lambda}_F(S)| < |S|$. Now we reverse the non-trivial segment, that is, rewrite the segment in the reverse order.
Then all intended pairs in $F$ are unchanged except for two intended pairs in $S$, while $|S|$ and $V(S)$ remain the same. Consequently, the degree in $F$ of every vertex in $Y_l$ does not decrease. Furthermore, $y''$ can gain a new edge from the newly created intended pair in $S$. This increases $|E(F)|$, contradicting the choice of $F$.

\begin{figure}[H]
	\centering
	\begin{subfigure}[b]{0.45\textwidth}
	\centering
	\begin{tikzpicture}[scale=0.9]
	\node at (120:2) {$\bullet$};
	\node at (90:2) {$\bullet$};
	\node at (60:2) {$\bullet$};
	\node at (-120:2) {$\bullet$};
	\node at (-90:2) {$\bullet$};
	\node at (-60:2) {$\bullet$};
	\draw [domain=-60:60] plot ({2*cos(\x)}, {2*sin(\x)});
	\draw [domain=120:240] plot ({2*cos(\x)}, {2*sin(\x)});
	\node [above] at (120:2) {$x_1$};
	\node [above] at (90:2) {$x_2$};
	\node [above] at (60:2) {$x_3$};
	\node [below] at (-60:2) {$x_4$};
	\node [below] at (-90:2) {$x_5$};
	\node [below] at (-120:2) {$x_6$};
	\end{tikzpicture}
	\end{subfigure}
    \qquad
    \begin{subfigure}[b]{0.45\textwidth}
    \centering
    \begin{tikzpicture}[scale=0.9]
    \node at (150:2) {$\bullet$};
    \node at (120:2) {$\bullet$};
    \node at (90:2) {$\bullet$};
    \node at (60:2) {$\bullet$};
    \node at (30:2) {$\bullet$};
    \draw [domain=150:390] plot ({2*cos(\x)}, {2*sin(\x)});
    \node [above] at (150:2) {$x_1$};
    \node [above] at (120:2) {$x_2$};
    \node [above] at (90:2) {$x_3=x_4$};
    \node [above] at (60:2) {$x_5$};
    \node [above] at (30:2) {$x_6$};
    \end{tikzpicture}
    \end{subfigure}
    \caption{Two configurations for $|S| = 4$}
    \label{fig:2-configs-for-4}
\end{figure}
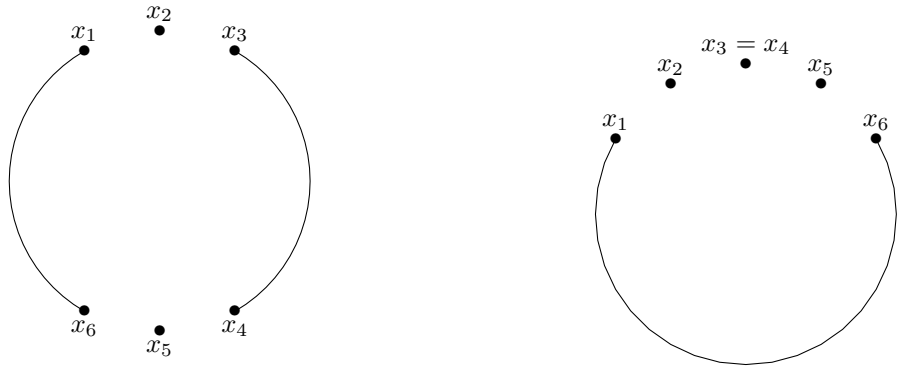

Suppose $|S| = 4$, and we consider our intended cycle as in the previous case. If all segments are isolated vertices, then $|X|=4$ and we are done by Theorem~\ref{thm:cycle-for-7}. If only one of the segments is an isolated vertex, or if exactly two segments are isolated vertices and they are consecutive, then $\mathsf{\Lambda}_F(S) = Y_l$, since every $y \in Y_l$ can miss at most three intended pairs in $S$. Therefore only two cases remain (see Fig. 2).
For convenience, let $V(S) = \{x_1,\dots,x_6\}$, where $x_2$ and $x_5$ are isolated vertices, each forming their own path. It is possible that $x_3 = x_4$ (see Fig. 2 (2)). Let $\mathsf{\Lambda}_{ij}=\mathsf{\Lambda}_F((x_i, x_j))$ be the set of vertices in $Y_l$ that are adjacent to the intended pair $(x_i, x_j)$ in $F$.

In both cases, to achieve $|\mathsf{\Lambda}_F(S)| < |S| = 4$, all but at most three vertices in $Y_l$ are non-adjacent to every intended pair in $S$. In other words, all but at most three vertices in $Y_l$ must be non-adjacent to both $x_2$ and $x_5$. By the double Hall property, we know $|\mathsf{\Lambda}_{ij}| \ge 2$ for each $(x_i, x_j) \in S$. Additionally, $|\mathsf{\Lambda}_{12} \cap \mathsf{\Lambda}_{45}| \ge 1$ since $|\mathsf{\Lambda}_{12} \cup \mathsf{\Lambda}_{45}| \le 3$. Now we view $x_2-x_3\cdots x_4$ as one segment and reverse it. By doing so, the intended pairs $(x_1, x_2)$ and $(x_4, x_5)$ are replaced by the pairs $(x_1, x_4)$ and $(x_2, x_5)$, respectively, while all other pairs remain unchanged. Thus we lose $|\mathsf{\Lambda}_{12}| + |\mathsf{\Lambda}_{45}|$ edges and gain $|\mathsf{\Lambda}_{14}| + |\mathsf{\Lambda}_{25}|$ edges. Since $\deg(y) \ge |X| - 2$ for any $y \in Y_l$, every vertex in $Y_l$ that is not adjacent to both $x_2$ and $x_5$ must be a neighbor of all other vertices in $X$. So $|\mathsf{\Lambda}_{14}| \ge |Y_l| - 3 + |\mathsf{\Lambda}_{12} \cap \mathsf{\Lambda}_{45}|$ and, by the double Hall property, $|\mathsf{\Lambda}_{25}| \ge 2$. Hence,
    \begin{align*}
           (|\mathsf{\Lambda}_{14}| + |\mathsf{\Lambda}_{25}|)-(|\mathsf{\Lambda}_{12}| + |\mathsf{\Lambda}_{45}|) & \ge (|Y_l| - 3 + |\mathsf{\Lambda}_{12} \cap \mathsf{\Lambda}_{45}|)+2-|\mathsf{\Lambda}_{12}|-|\mathsf{\Lambda}_{45}|\\
    & = |Y_l| - |\mathsf{\Lambda}_{12} \cup \mathsf{\Lambda}_{45}| - 1\\
    & \ge |Y_l| - 4.
    \end{align*}
Since $|Y_l| \ge m \ge |S| = 4$, we have $(|\mathsf{\Lambda}_{14}| + |\mathsf{\Lambda}_{25}|)-(|\mathsf{\Lambda}_{12}| + |\mathsf{\Lambda}_{45}|) \ge 0$, with equality holding if $|Y_l| = 4$, $|\mathsf{\Lambda}_{12} \cup \mathsf{\Lambda}_{45}| = 3$, $|\mathsf{\Lambda}_{25}| = 2$, and $|\mathsf{\Lambda}_{14}| = 1 + |\mathsf{\Lambda}_{12} \cap \mathsf{\Lambda}_{45}|$, where $|\mathsf{\Lambda}_{12} \cap \mathsf{\Lambda}_{45}| \ge 1$. In this case, let $Y_l = \{y_1, y_2, y_3, y_4\}$, and, without loss of generality, assume $y_1 \in \mathsf{\Lambda}_{12} \cap \mathsf{\Lambda}_{45}$, $y_2 \in \mathsf{\Lambda}_{25}$, $y_3 \notin \mathsf{\Lambda}_{25}$, and $y_4 \notin \mathsf{\Lambda}_G(x_2) \cup \mathsf{\Lambda}_G(x_5)$. Applying the double Hall property to the sets $\{x_1, x_2, x_5\}$ and $\{x_2, x_4, x_5\}$, respectively, indicates that $y_3$ satisfies $x_1, x_4 \in \mathsf{\Lambda}_G(y_3)$. Thus we know $y_1, y_3, y_4 \in \mathsf{\Lambda}_{14}$. Additionally, since $|\mathsf{\Lambda}_{14}| - 1 = |\mathsf{\Lambda}_{12} \cap \mathsf{\Lambda}_{45}| \le |\mathsf{\Lambda}_{25}| = 2$, it follows that $|\mathsf{\Lambda}_{14}| = 3$, implying $y_2 \notin \mathsf{\Lambda}_{14}$ and $|\mathsf{\Lambda}_{12} \cap \mathsf{\Lambda}_{45}| = 2$. However, our assumption implies $y_3, y_4 \notin \mathsf{\Lambda}_{12} \cap \mathsf{\Lambda}_{45}$, which forces $y_2 \in \mathsf{\Lambda}_{12} \cap \mathsf{\Lambda}_{45}$, i.e., $y_2 \in \mathsf{\Lambda}_{14}$, resulting in a contradiction. This case does not exist, and hence $(|\mathsf{\Lambda}_{14}| + |\mathsf{\Lambda}_{25}|)-(|\mathsf{\Lambda}_{12}| + |\mathsf{\Lambda}_{45}|) > 0$ always holds, contradicting the choice of $F$. 

Therefore $F$ satisfies Hall's condition, and by Theorem~\ref{hall-thm}, there exists a matching in $F$ that saturates all intended pairs. Consequently, there is a cycle in $G$ covering $X$.
\qed
 
\section{Degree bounds on dHp graphs}\label{section:degree-bound-on-dhp}

We will begin with the very short proof of the lower bound on the maximum degree in a dHp graph.

\begin{proof}[Proof of Theorem~\ref{thm:dhp-maximum-degree-bound}]
Take an arbitrary vertex $x \in X$. Such a vertex $x$ has neighbors $y_1, y_2, \dots, y_k$ with $k \leq d$. So there are at most $k(d-1) \leq d(d-1)$ edges from $y_1, \dots, y_k$ to the other vertices in $X$. However, for any $x' \in X\setminus \{x\}$, $|\mathsf{\Lambda}^2(\{x, x'\})| = |\mathsf{\Lambda}(x) \cap \mathsf{\Lambda}(x')| \geq 2$, i.e., each other vertex in $X$ must receive at least two of those edges. Since these edges can be partitioned into at most $\frac{d(d-1)}{2}$ subsets, each of which contains at least $2$ edges, there can be at most $\frac{d(d-1)}{2}$ vertices in $X\setminus \{x\}$, implying that $n \le \frac{d(d-1)}{2}+1=\binom d2 + 1$. It follows that $d \ge \frac12(1 + \sqrt{8n-7})$, which is at least $\sqrt{2n}$ for $n \ge 2$. 
\end{proof}

Next, we consider when we can match or approach this lower bound with a construction.

\subsection{Constructions via biplanes}\label{sec:biplanes}

Analyzing the proof above, we can see that this bound is tight when $n = \binom d2 + 1$. The general upper bound on $n=|X|$ is
\[
    |X| \le 1 + \frac12\sum_{y \in \mathsf{\Lambda}(x)} (\deg(y)-1)
\] 
which can only be equal to $\binom d2 + 1$ if $x$ and all its neighbors have degree $d$. Since $x$ is arbitrary, it follows that the graph $G$ must be $d$-regular (with the possible exception of isolated vertices in $Y$, which we ignore). Moreover, to have exactly $\binom d2$ other vertices in $X$, we need to have $|\mathsf{\Lambda}(x) \cap \mathsf{\Lambda}(x')| = 2$ for any $x' \ne x$: any two vertices in $X$ have exactly two common neighbors in $Y$. 

This means that $G$ has exactly $\binom n2$ cycles of length $4$: one for every pair of vertices in $X$. Since $G$ is regular, $|Y|=n$, so there are also $\binom n2$ pairs of vertices in $Y$, meaning that every pair of vertices in $Y$ is part of an average of one $4$-cycle in $G$: they have an average of two common neighbors in $X$. On the other hand, no two vertices $y_1, y_2 \in Y$ can have three common neighbors $x_1, x_2, x_3 \in X$; since no two of the vertices $x_1, x_2, x_3$ can have more common neighbors than $y_1, y_2$, we would have $\mathsf{\Lambda}^2(\{x_1, x_2, x_3\}) = \{y_1, y_2\}$, violating the double Hall property. Therefore any two vertices in $Y$ have exactly two common neighbors in $X$.

A graph with these properties is exactly the incidence graph of a symmetric block design with parameters $(v,k,\lambda) = (1 + \binom d2, d, 2)$: a block design with $1+\binom d2$ points and $1+\binom d2$ lines such that every pair of points lies on exactly two common lines, while every two lines intersect in exactly two points. A symmetric block design with $\lambda=2$ is known as a \textit{biplane}; the parameter $k-\lambda = d-2$ is called its \textit{order}.

Conversely, given a biplane, we can consider its incidence graph: the $(X,Y)$-bigraph where $X$ is the set of points, $Y$ is the set of lines, with an edge $xy$ exactly when point $x$ lies on line $y$. 

\begin{lemma}\label{lemma:biplane}
The incidence graph of a biplane always satisfies the double Hall property.
\end{lemma}
\begin{proof}
Take an arbitrary $S \subseteq X$ with $|S| \geq 2$. Let $y \in Y$ be a vertex maximizing $\deg_S(y)$, and let $d = \deg_S(y)$. Since every pair of points lies on exactly two common lines, we know $d \geq 2$, and all $\binom d2$ pairs of neighbors of $y$ in $S$ must have another common neighbor. Furthermore, these common neighbors cannot repeat, since that would result in two lines intersecting in three or more points, contradicting the property of the biplane. Therefore, together with $y$, there are at least $\binom d2+1$ vertices in $\mathsf{\Lambda}^2(S)$. So the double Hall property holds for $S$ if $\binom d2 + 1 \geq |S|$.

Otherwise, suppose $\binom d2 \leq |S|-2$. Since every pair of vertices in $S$ has exactly two common neighbors in $Y$, there are a total of $|S|(|S|-1)$ paths of length $2$ beginning and ending in $S$. Every midpoint of such a path is in $\mathsf{\Lambda}^2(S)$, and every vertex in $Y$ can be the midpoint of at most $\binom d2$ such paths. Therefore
\[
    |\mathsf{\Lambda}^2(S)| \geq \frac{|S|(|S|-1)}{\binom d2} \geq \frac{|S|(|S|-1)}{|S|-2} > |S|.
\]
In either case, we have $|\mathsf{\Lambda}^2(S)| \geq |S|$, completing the proof.
\end{proof}

There are only finitely many biplanes known, which give us the known examples where the bound of Theorem~\ref{thm:dhp-maximum-degree-bound} is tight. Marshall Hall in~\cite{book:biplane} lists several known examples of biplanes of orders $0, 1, 2, 3, 4, 7, 9, 11$, which correspond to regular graphs of degrees $2, 3, 4, 5, 6, 9, 11, 13$, respectively. Furthermore, the regular graphs of degrees $2, 3, 4$, and $5$ are each unique up to isomorphism. Notable special cases include the order-$1$ biplane, whose incidence graph is the cube graph, and the order-$2$ biplane, whose incidence graph is the bipartite complement of the incidence graph of the Fano plane.

We are most interested in order-$11$ biplanes, which give us $13$-regular dHp graphs with $79$ vertices in $X$ and $79$ vertices in $Y$.

\subsection{Bipartite product construction}\label{sec:bipartite-product} 

To bootstrap from the excellent but small biplane examples, we define a bipartite product that we show preserves the double Hall property.

\begin{definition}
Let $G=(X, Y)$ be a bipartite graph, and let $G' = (X', Y')$ be a bipartite graph. Then their bipartite product $G \times G'$ is defined as the bipartite graph with partitions $X \times X'$ and $Y \times Y'$, where vertices $(x,x')$ and $(y, y')$ are adjacent if and only if $xy \in E(G)$ and $x'y' \in E(G')$.
\end{definition}

\begin{lemma}\label{lemma:bipartite-product}
Suppose $G$ and $G'$ both satisfy the double Hall property. Then $G \times G'$ also satisfies the double Hall property.
\end{lemma}
\begin{proof}
Consider an arbitrary subset $S \subseteq X \times X'$ with $|S| \geq 2$. Our goal is to find at least $|S|$ vertices in $Y \times Y'$ with at least two neighbors in $S$. Let $T \subseteq X$ be the projection of $S$ onto $X$: the set of all $x \in X$ such that there is some $x' \in X'$ with $(x,x') \in S$. 

If $|T| = 1$, then $S$ is a subset of $\{x\} \times X'$ for some $x \in X$. Let $y$ be any neighbor of $x$ in $G$. Then the edges of $G \times G'$ between $\{x\} \times X'$ and $\{y\} \times Y'$ form a copy of $G'$. Since this copy of $G'$ satisfies the double Hall property, we know $|\mathsf{\Lambda}^2(S)| \ge |S|$.

If $|T| \ge 2$, then by Theorem~\ref{thm:hungarian-cycle-cover}, there exists a collection of disjoint cycles in $G$ whose union covers $T$. Pick an arbitrary orientation for each of those cycles. For each vertex $t \in T$, let:
\begin{itemize}
\item $y^+(t)$ be the vertex in $Y$ that follows $t$ on its cycle.
\item $x^+(t)$, accordingly, be the vertex in $X$ that follows $y^+(t)$ on its cycle.
\item $S(t) = S \cap (\{t\} \times X')$, so that the sets $S(t)$ for $t \in T$ form a partition of $S$.
\end{itemize}

For each $t$ such that $|S(t)| \ge 2$, the edges between $\{t\} \times X'$ and $\{y^+(t)\} \times Y'$ form a copy of $G'$. By the double Hall property in that copy of $G'$, there are at least $|S(t)|$ vertices in $\{y^+(t)\} \times Y'$ with at least two neighbors in $S(t)$.

For each $t$ such that $|S(t)| = 1$, let $(t, u)$ be the unique element of $S(t)$, and let $(x^+(t), v)$ be an arbitrary element of $S(x^+(t))$. Let $y \in Y'$ be a common neighbor of $u$ and $v$ in $G'$. Note that if $u \neq v$, then such $y$ exists by applying the double Hall property to $\{u, v\}$ in $G'$; if $v = u$, then just let $y$ be any neighbor of $u$. Then $(y^+(t), y)$ is a common neighbor of $(t,u)$ and $(x^+(t),v)$, so $(y^+(t),y) \in \mathsf{\Lambda}^2(S)$.

Thus, for each $t \in T$, we can find $|S(t)|$ vertices in $\{y^+(t)\} \times Y'$ that have at least two neighbors in $S$, and by putting these together, we learn that $|\mathsf{\Lambda}^2(S)| \ge |S|$.

In both cases, we have $|\mathsf{\Lambda}^2(S)| \ge |S|$, and therefore $G \times G'$ satisfies the double Hall property.
\end{proof}

Let $G$ be the incidence graph of an order-$11$ biplane. Then $G$ has the double Hall property by Lemma~\ref{lemma:biplane}, and by Lemma~\ref{lemma:bipartite-product}, this is also true of the iterated product $G^{\times k}$ for all $k\ge 1$.

This iterated product is an $(X,Y)$-bigraph with $|X|=|Y|=79^k$; it is regular of degree $d = 13^k$. Setting $n = 79^k$, we get $d = n^{\alpha}$ for $\alpha = \log_{79} 13 \approx 0.587$. Together with the small-$n$ constructions arising directly as the incidence graphs of biplanes, this completes the proof of Theorem~\ref{thm:dhp-maximum-degree-construction}.

\section{The double Hall property in random graphs}\label{section:dHp-random-graph}

To prove Theorem~\ref{thm:gnnp-is-dhp}, we follow the same overall approach as used by Erd\H{o}s and R\'enyi in \cite{erdosrenyi1964} (see section 6.1 of \cite{friezekaronski2015} for a modern presentation). 

We first consider the simplest obstacle to the double Hall property: pairs of vertices in $X$ with fewer than two common neighbors in $Y$. For these, we use the method of moments (see Theorem~23.11 in \cite{friezekaronski2015}) to prove that the number of these pairs follows a Poisson distribution.

\begin{lemma}\label{lemma:poisson-pairs}
Let $p = \sqrt{(2\log n + \log \log n + c)/n}$ for a constant $c$. In the random bipartite graph $\mathbb G_{n,n,p}$, let $\mathbf N$ be the number of pairs $\{x_1, x_2\} \subseteq X$ with at most one common neighbor in $Y$. Then $\mathbf N$ converges in distribution to a Poisson random variable with rate $e^{-c}$: for all $k \ge 0$, as $n \to \infty$, $\Pr[\mathbf N = k] \to e^{-e^{-c}} \cdot \frac{e^{-ck}}{k!}$.
\end{lemma}

This will establish the correctness of Theorem~\ref{thm:gnnp-is-dhp} at the threshold value of $p$ once we prove that with high probability (that is, with probability tending to $1$ as $n \to \infty$), no other obstacles to the double Hall property occur at the threshold value. We must be precise in what we mean by ``other obstacles'', since a pair $\{x_1, x_2\}$ of the type in Lemma~\ref{lemma:poisson-pairs} might be the root cause of a larger subset of $X$ that has too few common neighbors. 

To this end, we define an \emph{obstacle} (to the double Hall property) to be  a pair $(S,T)$ with $S \subseteq X$ and $T \subseteq Y$ such that $|S| > |T|$, $|S|\ge 2$, and $T$ contains all vertices of $Y$ with at least two neighbors in $S$. We define an obstacle $(S,T)$ to be \emph{minimal} if there is no obstacle $(S',T')$ with $S' \subseteq S$, $T' \subseteq T$, and $|S'| + |T'| < |S|+|T|$.

\begin{lemma}\label{lemma:no-big-obstacles}
Let $p = \sqrt{(2 \log n)/n}$. In the random bipartite graph $\mathbb G_{n,n,p}$, with high probability, there are no minimal obstacles $(S,T)$ with $|S| \ge 3$.
\end{lemma}

Once there are no obstacles $(S,T)$ with $|S|=2$, and no minimal obstacles $(S,T)$ with $|S|\ge 3$, it follows that there are no obstacles at all.

The double Hall property is monotonic in the sense that adding more edges to a dHp graph will always produce another dHp graph. As a result, if Lemma~\ref{lemma:no-big-obstacles} holds when $p = \sqrt{(2 \log n)/n}$, it holds for all larger $p$ as well. 
In particular, we can conclude that when $p = \sqrt{(2\log n + \log \log n + c)/n}$, the conclusion of Theorem~\ref{thm:gnnp-is-dhp} holds. The rest of the theorem follows, once again, by monotonicity, since $\Pr[\mathbb G_{n,n,p} \text{ is dHp}]$ is increasing in $p$.

In the proofs in this section, we apply the following standard Chernoff-type bound:

\begin{lemma}[e.g.\ Corollary 24.7 in~\cite{friezekaronski2015}]\label{lemma:chernoff}
Let $\mathbf X$ be a binomial random variable with mean $\mu$. Then for all $0 \le \delta \le 1$, 
\begin{align*}
    \Pr[\mathbf X \ge (1+\delta)\mu] &\le \exp\left(-\frac13 \delta^2 \mu\right)\\
    \Pr[\mathbf X \le (1-\delta)\mu] &\le \exp\left(-\frac12 \delta^2 \mu\right)
\end{align*}
\end{lemma}

Lemma~\ref{lemma:chernoff} also lets us bound the maximum degree $np$ of $\mathbb G_{n,n,p}$ at the threshold. The degree of a vertex is binomial with mean $np$; taking $\delta = 3(\frac{\log n}{n})^{1/4}$, we see that the probability that a vertex has degree more than $(1+\delta)np$ is less than $\exp(-\frac13 \delta^2 \sqrt{2n \log n}) = \exp(-3\sqrt2 \log n) = n^{-3\sqrt2}$. Thus, the expected number of vertices with degree more than $(1+\delta)np$ goes to $0$ as $n \to \infty$. Since $\delta \to 0$ as $n \to \infty$ as well, with high probability the maximum degree is asymptotic to $np \sim \sqrt{2n \log n}$.

\subsection{Proof of Lemma~\ref{lemma:poisson-pairs}}

We write the random variable $\mathbf N$ of Lemma~\ref{lemma:poisson-pairs} as $\mathbf N_0 + \mathbf N_1$, where $\mathbf N_0$ is the number of pairs $\{x_1,x_2\} \subset X$ with no common neighbors in $Y$, and $\mathbf N_1$ is the number of triples $\{x_1, x_2, y\}$ with $x_1, x_2 \in X$ and $y \in Y$ such that $x_1$ and $x_2$ are both adjacent to $y$ but have no other common neighbors.

For two fixed vertices $x_1, x_2$, the probability that they have no common neighbors is $(1-p^2)^n \le e^{-np^2} = \frac1{n^2} \cdot \frac1{\log n} \cdot e^{-c}$. There are $\binom n2 \le \frac{n^2}{2}$ choices of $\{x_1, x_2\}$, so $\mathbb E[\mathbf N_0] \le e^{-c}/(2 \log n)$, and with high probability $\mathbf N_0 = 0$. Therefore it is enough to focus on $\mathbf N_1$ for the remainder of the proof.

To analyze $\mathbf N_1$, we compute $\mathbb E[\mathbf N_1(\mathbf N_1-1) \cdots (\mathbf N_1-t+1)]$ for a constant $t$: the expected number of sequences of $t$ distinct triples $\{x_{i1}, x_{i2}, y_i\}$ with $1 \le i \le t$ such that for all $i$, $y_i \in Y$ is the unique common neighbor of $x_{i1}, x_{i2} \in X$. If this expected value is asymptotic to $e^{-ct}$, then by the method of moments, we may conclude that $\mathbf N_1$ is asymptotically Poisson with rate $e^{-c}$, proving the lemma.

We employ the usual approach to computing the expected value: for each sequence of $t$ distinct triples $\{x_{i1}, x_{i2}, y_i\}$, we compute the probability of the event that $y_i$ is the unique common neighbor of $x_{i1}$ and $x_{i2}$, and sum over all these probabilities. We discard sequences for which this event is automatically impossible, such as ones in which two vertices in $X$ appear together in two triples.

Given such a sequence, let $H$ be the graph formed by the vertices in the triples and the required edges $x_{i1} y_i, x_{i2} y_i$ for $i=1, \dots, t$. The vertices $x_{i1}, x_{i2}$ have degree at least $1$ in $H$, and the vertices $y_i$ have degree at least $2$; therefore, if $H$ has $m$ edges, then it has at most $m/2$ distinct vertices in $Y$ and at most $m$ distinct vertices in $X$.

We begin by ruling out all cases where $m < 2t$. Here, the expected number of graphs $H$ of this type that appear in the random graph is at most $n^{3m/2} p^m$, and the number of sequences of triples that result in a particular graph $H$ is bounded by a constant dependent on $t$. For a given vertex $y \in Y$ other than $y_1, \dots, y_t$, let $A_{iy}$ be the event that $y$ is adjacent to both $x_{i1}$ and $x_{i2}$. By Bonferroni's inequalities, we can get an upper bound for the probability that none of the $A_{iy}$ occur by truncating the inclusion-exclusion sum at the level of pairwise overlaps:
\[
    \Pr\left[\bigwedge_{i=1}^t \neg A_{iy}\right] \le 1 - \sum_{i=1}^t \Pr[A_{iy}] + \sum_{i=1}^t \sum_{j=i+1}^t \Pr[A_{iy} \land A_{jy}].
\]
For all $i$, $\Pr[A_{iy}] = p^2$, while $\Pr[A_{iy} \land A_{jy}]$ is $p^4$ if $\{x_{i1}, x_{i2}\}$ is disjoint from $\{x_{j1}, x_{j2}\}$, and $p^3$ if not. Therefore the right-hand side of the inequality is at most $1 - t p^2 + \binom t2 p^3$. The event $\bigwedge_{i=1}^t \neg A_{iy}$ occurs independently for at least $n-t$ vertices in $Y$ and outside $H$, so the probability that no pair $x_{i1}, x_{i2}$ has a common neighbor other than $y_i$ is at most
\begin{align*}
    \left(1 - t p^2 + \binom t2 p^3\right)^{n-t} &\le \exp\left(-t(n-t)p^2 + \binom t2 (n-t) p^3\right)\\
     & = e^{-tnp^2} \cdot e^{t^2 p^2} \cdot e^{(n-t)t(t-1)p^3/2}.
\end{align*}
The last two factors $e^{t^2 p^2}$ and $e^{(n-t)t(t-1)p^3/2}$ are both asymptotic to $1$, since $p^2 \to 0$ and $np^3 \to 0$ as $n \to \infty$, while $t$ is a constant. Therefore the probability is asymptotic to $e^{-tnp^2}$, which is less than $e^{-t(2\log n)} = n^{-2t}$ for sufficiently large $t$.

In conclusion, the expected number of sequences of triples that result in an $m$-edge graph $H$ is $O(n^{3m/2} p^m n^{-2t})$, which is $n^{3m/2 - m/2 - 2t}=n^{m - 2t}$ up to a poly-logarithmic factor. For each $m<2t$, this expected value goes to $0$ as $n \to \infty$.

The remaining contribution to $\mathbb E[\mathbf N_1(\mathbf N_1-1) \cdots (\mathbf N_1-t+1)]$ arises from sequences of triples $\{x_{i1}, x_{i2}, y_i\}$ where all $2t$ edges are distinct. The number of ways to choose the vertices for such a sequence is less than $\binom n2^t \cdot n^t$ and at least $\frac{n!}{(n-2t)! 2^t} \cdot \frac{n!}{(n-t)!}$, both of which are asymptotic to $n^{3t}/2^t$. For each such sequence, the probability that for all $i$, $y_i$ is the unique common neighbor of $x_{i1}$ and $x_{i2}$ is $p^{2t} (1-p^2)^{t(n-1)}$, so 
\begin{align*}
\mathbb E[\mathbf N_1(\mathbf N_1-1) \cdots (\mathbf N_1-t+1)] 
    &\sim \frac{n^{3t}}{2^t} p^{2t} (1-p^2)^{t(n-1)} \\
    &\sim \frac{n^{3t}}{2^t} \cdot \frac{(2 \log n)^t}{n^t} \cdot e^{-tnp^2} \\
    &= n^{2t} (\log n)^t e^{-2t \log n - t \log \log n - c t} \\
    &= e^{-ct},
\end{align*}
as desired. \hfill\qed

\subsection{Proof of Lemma~\ref{lemma:no-big-obstacles}}

In Lemma~\ref{lemma:no-big-obstacles}, we show that for $p = \sqrt{(2 \log n)/n}$, with high probability, $\mathbb G_{n,n,p}$ contains no minimal obstacles $(S,T)$ with $|S| \ge 3$. For $|S|\ge 3$, if an obstacle is minimal, then in particular $|T| = |S|-1$; otherwise, we may obtain a smaller obstacle by removing any element from $S$.

Our argument is to show that the expected number of pairs $(S,T)$ that are minimal obstacles goes to $0$ as $n \to \infty$. We let $k = |S|$, so that $|T|=k-1$, and consider three ranges of $k$, for which we give separate arguments. Note that these three ranges do cover all possible values of $k$ for sufficiently large $n$.

\textbf{Case 1.} $3 \le k \le n/10$.

By minimality of the obstacle $(S,T)$, each vertex in $T$ must have at least $3$ neighbors in $S$: if there is a $y \in T$ with neighbors only $x_1,x_2 \in S$, then we obtain a smaller example by excluding $x_1$ from $S$ and $y$ from $T$. Therefore there are at least $3(k-1)$ edges between $S$ and $T$.

There are $\binom nk \binom n{k-1}$ potential choices of $(S,T)$ for each value of $k$; for each choice, the probability that there are at least the required $3(k-1)$ edges between $S$ and $T$ is at most $\binom{k(k-1)}{3(k-1)} p^{3(k-1)}$. 

For each $y \in Y \setminus T$, the probability that $y$ has at most one neighbor in $S$ is $(1-p)^k + kp(1-p)^{k-1}$, which we bound as follows:
\begin{align*}
      (1-p)^k + kp(1-p)^{k-1} & = (1-p)^{k-1}(1 - p + kp)\\
       & \le (1-p)^{k-1}(1+p)^{k-1}\\
       & \le e^{-(k-1)p^2}.  
\end{align*}
There are $n-k+1$ vertices $y$, for which the chances of having at most one neighbor in $S$ are independent. By the case, $n-k+1 \ge 0.9n$, so the probability that no vertex in $Y \setminus T$ has more than one neighbor in $S$ is at most $e^{-0.9n(k-1)p^2}$;
since $p =\sqrt{\frac{2 \log n}{n}}$, this simplifies to $n^{-1.8(k-1)}$.

We conclude that the expected number of minimal obstacles $(S,T)$ for a particular value of $k$ is at most
\[
    \binom nk \binom n{k-1} \binom{k(k-1)}{3(k-1)} p^{3(k-1)} n^{-1.8(k-1)}.
\]
Taking $\binom n{k-1} \le (\frac{ne}{k-1})^{k-1}$, $\binom nk \le (\frac{ne}{k})^{k}$,  and $\binom{k(k-1)}{3(k-1)} \le (\frac{k(k-1)e}{3(k-1)})^{3(k-1)} = (\frac{ke}{3})^{3(k-1)}$, 
we get a further upper bound of
\[
    \left(\frac{ne}{k}\right)^k \left(\frac{ne}{k-1}\right)^{k-1} \left(\frac{ke}{3}\right)^{3(k-1)} p^{3(k-1)} n^{-1.8(k-1)} < n \left( \frac{n^2 e^2}{k(k-1)} \cdot \frac{k^3 e^3}{27} \cdot p^3 \cdot n^{-1.8}\right)^{k-1}.
\]
First, we inspect this upper bound for small values of $k$, specifically $3\le k\le 7$. Here, the base of the exponent is $O(n^{0.2} p^3) = O(n^{-1.3} (\log n)^{1.5})$. Thus it is always bounded by a function that goes to $0$ as $n \to \infty$ and, in particular, is strictly less than $1$ for sufficiently large $n$. Consequently, the expected number of minimal obstacles for a fixed $k$ in this range is bounded by $O(n(n^{-1.3} (\log n)^{1.5})^2)=O(n^{-1.6} (\log n)^3)$. 

Second, over the range of $8 \le k \le n/10$, we have
\[
    \frac{n^2e^2}{k(k-1)} \cdot \frac{k^3e^3}{27} \cdot p^3 \cdot n^{-1.8} = O(n^{1.2} p^3) = O(n^{-0.3} (\log n)^{1.5}).
\]
Similarly, the base of the exponent is strictly less than $1$ for sufficiently large $n$.
This means that the expected number of minimal obstacles for a fixed $k$ in this range is at most $O(n(n^{-0.3} (\log n)^{1.5})^7)=O(n^{-1.1} (\log n)^{10.5})$. Thus the total expected number of minimal obstacles $(S, T)$ with $3 \le k \le n/10$ is at most $O(n^{-0.1} (\log n)^{10.5}) + O(n^{-1.6} (\log n)^3) = O(n^{-0.1} (\log n)^{10.5})$, which goes to $0$ as $n \to \infty$. Therefore, with high probability, no minimal obstacle falling under Case 1 occurs in $\mathbb G_{n,n,p}$.

\textbf{Case 2.} $k \ge 32/p$ and $n-k \ge 32/p$.

In this case, we observe that there must be at most $n-k+1$ edges between $S$ and $Y \setminus T$, but that number of edges is distributed binomially with mean $\mu = pk(n-k+1)$; in particular, $n-k+1 < \frac12 \mu$, so a Chernoff bound (as in Lemma~\ref{lemma:chernoff}) gives an upper bound on the probability of $\exp(-\mu/8) = \exp(-pk(n-k+1)/8)$.

Given the case, we can say either that $pk \ge 32$ and $n-k+1 \ge n/2$, or that $p(n-k+1) \ge 32$ and $k \ge n/2$; in either case, $\exp(-pk(n-k+1)/8) \le \exp(-32(n/2)/8) = e^{-2n}$. There are fewer than $4^n$ pairs $(S,T)$, so the expected number of them falling under this case that are minimal obstacles is at most $4^n e^{-2n}$. Therefore, with high probability, no minimal obstacle falling under Case 2 occurs in $\mathbb G_{n,n,p}$.

\textbf{Case 3.} $n-k \le n/(4e)$.

The degree distribution of vertices in $Y$ is binomial with mean $np$, so by Lemma~\ref{lemma:chernoff}, the probability that a vertex has degree $\frac23 np$ or less is at most $\exp(-np/18)$. The expected number of such vertices in $Y$ is $n \exp(-np/18)$, which goes to $0$ as $n \to \infty$, so with high probability all vertices in $Y$ have degree at least $2np/3$; in particular, for large enough $n$, this is at least $np/2 + 1$.

Let $j = n-k+1$. If a pair $(S,T)$ as in the case is a minimal obstacle, then each of the $j$ vertices in $Y \setminus T$ can have at most one neighbor in $S$, so each of these vertices must have at least $np/2$ neighbors among the $j-1$ vertices in $X \setminus S$. In particular, there must be at least $jnp/2$ edges between $X \setminus S$ and $Y \setminus T$. The expected number of times this can occur for a given value of $j$ is bounded by
\begin{align*}
        \binom nj \binom n{j-1} \binom{j(j-1)}{jnp/2} p^{jnp/2} & \le n^{2j-1} \left(\frac{j(j-1)e}{jnp/2}\right)^{jnp/2} p^{jnp/2}\\
         & \le  n^{-1} \left(n^2 \left(\frac{2e(j-1)}{n}\right)^{np/2} \right)^j.
\end{align*}

By the case, $2e(j-1)/n \le 1/2$, and for sufficiently large $n$, $n^2 (1/2)^{np/2} \le 1/2$ as well, so the expected number of pairs $(S,T)$ satisfying the case for each $j$ is bounded by $n^{-1}/2^j$. We can now take a sum, not just over all $j$ in the case, but over all $j\ge 0$, and still get a quantity which goes to $0$ with $n$; therefore, with high probability, no minimal obstacle falling under Case 3 occurs in $\mathbb G_{n,n,p}$. \hfill\qed

\section*{Acknowledgments}
GC was partially supported by NSF grant DMS-2154331.

\section*{Conflict of Interest Statement}
The authors declare that they have no conflict interest regarding
the publication of this paper.

\bibliographystyle{plain}
\bibliography{Reference.bib} 

\end{document}